\documentclass{article}
\usepackage{geometry}
\usepackage{authblk}
\usepackage[numbers]{natbib}
\usepackage{amssymb}
\usepackage{amsmath}
\usepackage[linesnumbered, ruled]{algorithm2e}
\usepackage[hyphens]{url}
\usepackage[breaklinks]{hyperref}
\usepackage{enumitem}   
\usepackage{underscore} 
\usepackage{subcaption}
\usepackage{doi}

\usepackage{mathrsfs}
\usepackage{bm}

\usepackage{mathtools}
\mathtoolsset{showonlyrefs}

\usepackage{dirtytalk}
\usepackage{siunitx}

\let\bf\mathbf
\let\fr\frac
\let\eps\varepsilon

\newcommand{\One}{\mathds{1}}
\newcommand{\aplus}{a}
\newcommand{\astart}{a}
\newcommand{\aend}{b}

\newcommand{\spline}[1]{S^{#1}_{\mathcal{A}}}

\newcommand{\supnorm}[3]{\sup_{#1 \in #2} \left|#3 \right| }

\newcommand{\dto}{\, \mathrm{d}}

\newcommand{\der}[2]{\frac{\dto #1}{\dto #2}}

\newcommand{\hilfder}[3]{{_#1}\mathrm{D}_{#2}^{#3}}
\newcommand{\rlint}[3]{{_#1}\mathrm{I}_{#2}^{#3}}
\newcommand{\splint}[3]{{_#1^q}\Phi_{#2}^{#3}}

\newcommand{\Lim}[1]{\raisebox{0.5ex}{\scalebox{0.8}{$\displaystyle \lim_{#1}\;$}}}

\usepackage{aligned-overset}

\usepackage{amsthm}
\usepackage{dsfont}

\newtheorem{theorem}{Theorem}[section]
\newtheorem{lemma}{Lemma}[section]
\newtheorem{proposition}{Proposition}[section]

\theoremstyle{definition}
\newtheorem{definition}{Definition}[section]

\theoremstyle{remark}
\newtheorem{remark}{Remark}[section]

\newcommand{\rlder}[3]{{_#1}\mathrm{D}_{#2}^{#3}}

\newcommand{\weighnorm}[1]{\|#1\|_{C_{1-\gamma}[0,T]}}

\newcommand{\epsweighnorm}[1]{\|#1\|_{C_{1-\gamma}[\eps,T]}}

\newcommand{\Cnorm}[3]{\|#3\|_{C[#1, #2]}}

\newcommand{\Rbb}{\mathbb{R}}

\newcommand{\supnum}[1]{{(#1)}}

\title{A numerical Bernstein splines approach for nonlinear initial value problems with Hilfer fractional derivative} 
\date{\today}
\author[1,2]{Niels Goedegebure\footnote{\url{goed0001@e.ntu.edu.sg}}}
\author[1]{Kateryna Marynets\footnote{\url{k.marynets@tudelft.nl} (corresponding author)}}
\affil[1]{\small Delft Institute of Applied Mathematics, Faculty of Electrical Engineering, Mathematics and Computer Science, Delft University of Technology. Mekelweg 4, 2628CD Delft, The Netherlands
}
\affil[2]{\small Division of Mathematical Sciences, School of Physical and Mathematical Sciences, Nanyang Technological University, 21 Nanyang Link, Singapore 637371, Singapore\footnote{Present address}
}

\begin{document}

\maketitle
\begin{abstract}
{\normalsize \noindent \centering
The Hilfer fractional derivative generalizes and interpolates between the commonly used Riemann-Liouville and Caputo fractional derivative.
In general, solutions to Hilfer fractional derivative initial value problems are singular for $t\downarrow t_0$, providing a challenge in numerical approximation.
In this work, we present a Bernstein splines method to obtain accurate approximations to solutions of nonlinear Hilfer fractional derivative initial value problems, including solutions with singular behavior.
We provide explicit convergence requirements and asymptotic convergence rates to analytical solutions.
Moreover, the method is efficiently implemented, using a vectorized approach and parallelization in polynomial order.
Numerical experiments show empirical convergence corresponding to analytical results and a significantly higher accuracy compared to the commonly used fractional Adams–Bashforth–Moulton predictor-corrector method.
Furthermore, our method is able to simulate nonlinear systems with Hilfer fractional derivatives, as demonstrated by the application to the fractional Van der Pol oscillator.
}
\end{abstract}


\section{Introduction}
\label{sec:intro}
Fractional calculus offers  extension of classical differential and integral operators to non-integer orders of differentiation and integration.
Particularly of interest are fractional derivatives, generalizing $\alpha$-order differentiation for $\alpha \in \mathbb{R}_{>0}$.
Fractional differential equations (FDEs) involving these fractional derivatives have gained considerable interest from applications in control theory \cite{Podlubny1999control, Calderon2006control, Caponetto2010} and physical systems \cite{Hilfer2000, Hilfer2002, METZLER20001}.
A general challenge for the modeling and interpretation of FDEs is the choice of initial conditions.
For many common choices of fractional differential operators, FDEs have non-integer, fractional order initial values, as discussed in e.g. \cite{Podlubny1999}.
In the Hilfer fractional differential operator of order $\alpha$ and type $\beta \in [0,1]$, the choice of operator and initial conditions is parametrized by interpolating between two commonly used fractional differential operators: Riemann-Liouville for $\beta = 0$ and Caputo for $\beta = 1$ \cite{Hilfer2000}.
These two operators and their properties are extensively discussed in for instance \cite{Podlubny1999, Kilbas2006}.
The Caputo derivative operator yields integer order initial conditions, whereas the Riemann-Liouville operator and all other choices of $\beta \in (0,1)$ result in fractional order initial conditions.

In this paper, we consider the following nonlinear fractional initial value problem (FIVP):
\begin{align}
    \begin{cases}
        \hilfder{0}{t}{{\alpha}, {\beta}}
        \bf{y}(t) 
        =
        \bf{f}(t, \bf{y}(t)),
        \quad &0 < t \leq T,
        \quad 
        0< \alpha <1,
        \quad 
        0\leq \beta \leq 1,
        \\
        \rlint{0}{t}{1-{\gamma}} \bf{y}(0)
        =
        \bf{\tilde{y}}_0,
        \quad &\gamma = \alpha + \beta - \alpha \beta,
        \label{eq:ivp}
    \end{cases}
\end{align}
where $\bf{y}: (0, T] \to \mathbb{R}^d$ are the unknown vector-valued solution functions for initial values $\bf{\tilde{y}}_0 \in \mathbb{R}^d$ and $\bf{f}: (0, T] \times \mathbb{R}^d \to \mathbb{R}^d $ is a given vector-valued function, in general nonlinear with respect to both of its arguments.
Here, $\rlder{0}{t}{\alpha, \beta}$ denotes the Hilfer fractional derivative of order  $\alpha$ and type $\beta$ as introduced in \cite{Hilfer2000} and
$\rlint{0}{t}{1-\gamma}$ denotes the fractional Riemann-Liouville integral. 
In this paper, operations $\rlint{0}{t}{\alpha}$, $\hilfder{0}{t}{\alpha, \beta}$, $| \cdot |$ and $\leq$ are applied component-wise.

One of the known results concerning analysis of the FIVP \eqref{eq:ivp} is of \citet{Furati2012}, where the authors present conditions for existence and uniqueness of solutions to the studied problem.
For the general Hilfer derivative case of $\beta \not= 1$ with nonhomogeneous initial value $\bf{\tilde{y}}_0 \not  = \bf{0}$, these solutions are noncontinuous and singular as $t \downarrow 0$. 
This poses challenges in implementation and convergence for the numerical approximations to solutions compared to the Caputo case of $\beta = 1$, for which numerous commonplace established methods exist \cite{Diethelm2002, Samko1987, charef1992}.

A numerical-analytical operational method for solutions of the linear case of \eqref{eq:ivp} was presented by \citet{Hilfer2009operational}.
More recently, \citet{Admon2023} propose a method for nonautonomous linear Hilfer-derivative relaxation-oscillations, applying integration on both sides of \eqref{eq:ivp} and taking an operational matrix approach with Legendre functions.
Their method shows results for nonzero initial conditions. 
More recent work on nonlinear equations include the work by \citet{George2023}, obtaining \textit{mild} solutions of \eqref{eq:ivp}.
These solutions however, are restricted to the space of continuously differentiable solutions.
Moreover, the presented numerical examples only show results for the nonsingular case of $\tilde{\bf{y}}_0 = \bf{0}$.
A shifted Legendre-Galerkin approach applied to delay equations is taken by \citet{Sweis2023}, showing convergence under the assumptions of $L^2$-solutions, limiting the applications for more strongly singular problems.
Moreover, no demonstrated singular solution behavior is provided.
Under similar assumptions, \citet{Shloof2022} present a Legendre operational matrix approach for a variation of \eqref{eq:ivp} using the Hilfer \textit{fractal-fractional} operator.
In the recent paper by \citet{Muthuselvan2025}, a Legendre wavelet approach is taken for Hilfer-derivative FIVP's applied to delay equations.
Convergence of the method relies on continuity of solutions, and numerical examples are only shown for the continuous, nonsingular case of homogeneous initial values.

To the best of our knowledge, no efficient method with theoretical error bounds and convergence results exists in literature to obtain the full range of - possibly singular - solutions of \eqref{eq:ivp}.
Therefore, in this paper, we present a numerical method to obtain these solutions, building on the analysis of the author's thesis work \cite{Goedegebure2025thesis}.
Here, a Bernstein splines method is used, inspired by the recent work of \cite{Satmari2022,Bica2023,Satmari2023}, where only the Caputo case $(\beta=1)$ is considered.
Moreover, in comparison to these results, our method as presented in this paper requires less stringent convergence assumptions, removing constraints on time domain interval size $T$ by using local knot-for-knot iterations.
To obtain accurate approximations to the singular solutions of \eqref{eq:ivp} for the general Hilfer case $(0\leq \beta <1$), we apply a time-domain transformation of solutions using a {weighted function space}-approach as in the $C_{1-\gamma}$ space used by \citet{Furati2012}, together with a shifted time domain transformation resulting in solutions starting from $\eps > 0$ small.
Our method provides accurate approximations with asymptotic convergence rates under relaxed convergence assumptions for a wide range of nonlinear problems.
Additionally, we present a vectorized and parallelized implementation.

The outline of the paper is as follows. 
In Section \ref{sec:prob}, we provide preliminary results used in the rest of the paper.
We first present existing notation and results on fractional calculus in Section~\ref{sec:fr_calc}, establishing existence and uniqueness to analytical solutions of~\eqref{eq:ivp}.
In Section \ref{sec:prelim}, we present preliminary Bernstein polynomial and spline approximation results.
In Section \ref{sec:spline_conv}, the main convergence results on Hilfer Bernstein splines for FIVP's are presented in Theorem~\ref{thm:hilf_splines} and \ref{thm:convergence_rates}, with a detailed numerical implementation in Section \ref{sec:implementation}.
Numerical simulation examples and results are presented in Section \ref{sec:num_res}.
We show empirical convergence for a test problem in Section \ref{sec:num_res_conv} and a comparison with the Adams–Bashforth–Moulton predictor-corrector method of \citet{Diethelm2002} in Section~\ref{sec:num_res_diethelm}.
Finally, we present a nonlinear example in Section~\ref{sec:num_res_vdp}, showing simulations for the Hilfer-fractional Van der Pol oscillator, a system with applications in control theory and electronic circuits as studied in e.g. \cite{Barbosa2007, Chen2008,Liu2016, Liu2016vibcont}.
Conclusions, discussion of results and avenues for future work are presented in Section~\ref{sec:conclusion}.
\section{Preliminary definitions and results}
\label{sec:prob}
\subsection{Fractional calculus}
\label{sec:fr_calc}
\begin{definition}[Riemann-Liouville fractional integral, \cite{Samko1987}]\label{def:rl_int}
For $\alpha \geq 0$ and $a \leq t$, define the Riemann-Liouville fractional integral as:
\begin{align}
    \rlint{a}{t}{\alpha} f(t)
    =
    {\frac {1}{\Gamma(\alpha)}}\int_{a}^{t}\left(t-s\right)^{\alpha-1}f(s)\,\mathrm {d} s.
\end{align}
Furthermore, we use the notational convention $\rlint{a}{t}{\alpha} f(a) := \Lim{t\downarrow a} \rlint{a}{t}{\alpha} f(t)$.
\end{definition}
This operator has a number of useful properties: 
\begin{proposition}[Semigroup property of fractional integration, \cite{Samko1987}]
    \label{prop:rlint_semigroup}
    For $\alpha, \beta > 0$ and $a\leq t\leq b$, if $f \in L^p[a, b]$ with $1\leq p \leq \infty$, we have:
    \begin{align}
    \label{eq:semigroupeq}
        (
        \rlint{a}{t}{\alpha}
        \rlint{a}{t}{\beta}
        )
        f(t)
        =
        \rlint{a}{t}{\alpha + \beta}
        f(t),
    \end{align}
    almost everywhere.
    Furthermore, if $\alpha+\beta>1$, equality \eqref{eq:semigroupeq} holds pointwise.
\end{proposition}
The following boundedness and continuity result will be useful for further results.
\begin{proposition}[$L^p$ boundedness in $f$, \cite{Samko1987}] \label{prop:rl_bdd} 
The Riemann-Liouville integral is bounded in $L^p[a,b]$ for $f \in L^p[a,b]$, with $1\leq p \leq \infty $, satisfying:
\begin{align}
    ||\rlint{a}{t}{\alpha}f||_{L^p[a,b]}
    \leq 
    \fr{(b-a)^\alpha}{\Gamma(\alpha+1)}
    ||f||_{L^p[a,b]}.
\end{align}
By linearity of integration, this result also provides continuity in $L^p[a,b]$ with regards to $f$.
\end{proposition}
We introduce two special functions:
\begin{definition}[Beta function, \cite{Samko1987}]
\label{def:intro_beta}
For $a, b> 0$, denote the (complete) beta function as:
\begin{align}
    \mathrm{B}(a,b) 
    =
    \int_0^1
    \vartheta^{a-1}(1-\vartheta)^{b-1}\dto \vartheta.
\end{align}
\end{definition}
\begin{definition}[Incomplete beta function, \cite{Osborn1968}]
\label{def:intro_beta_incomplete}
For $a, b> 0$ and $0\leq z \leq 1$, denote the incomplete beta function as:
\begin{align}
    \mathrm{B}_z(a,b) 
    =
    \int_0^z
    \vartheta^{a-1}(1-\vartheta)^{b-1}\dto \vartheta.
\end{align}
\end{definition}
We can now use the beta functions as introduced above to explicitly calculate the fractional integral of a locally supported monomial.
\begin{proposition}[Fractional integral of monomials]
    \label{prop:int_pol_inc_0_t}
    For $\alpha >0$, $k > -1$ and $0\leq b\leq t$, the $\alpha$-fractional integral on $[0,b]$ of a monomial of degree $k$ is given by
    \begin{align}
    \rlint{0}{b}{\alpha}
    t^k
    =
    \fr{t^{\alpha+k}}{\Gamma(\alpha)}
    \mathrm{B}_{b/t}(k+1, \alpha).
    \end{align}
\end{proposition}
\begin{proof}
Substituting $z:=\fr{s}{t}$ gives:
    \begin{align}
        \fr{1}{\Gamma(\alpha)}
    \int_0^b(t-s)^{\alpha-1}s^k\dto s
    =
    \fr{t^{\alpha+k}}{\Gamma(\alpha)}
    \int_0^{b/t}(1-z)^{\alpha-1}z^k\dto z
    =
    \fr{t^{\alpha+k}}{\Gamma(\alpha)}
    \mathrm{B}_{b/t}(k+1, \alpha).
    \end{align}
\end{proof}
\begin{proposition}[Fractional integral of monomial with right-local support]
    \label{prop:int_pol_inc_t_T}
    For $\alpha >0$, $k > -1$ and $0\leq b\leq t$, the $\alpha$-fractional integral on $[b, t]$ of a monomial of degree $k$ is given by
    \begin{align}
    \rlint{b}{t}{\alpha}
    t^k
    =
    \fr{t^{\alpha+k}}{\Gamma(\alpha)}
    \mathrm{B}_{1-b/t}(\alpha, k+1).
    \end{align}
\end{proposition}
\begin{proof}\
First, using the substitution $\varphi := t-s$ gives
    \begin{align}
    \rlint{b}{t}{\alpha}
    t^k
    =
    \fr{1}{\Gamma(\alpha)}
    \int_b^t (t-s)^{\alpha-1}s^k\dto s
    =
    \fr{1}{\Gamma(\alpha)}
    \int_0^{t-b} \varphi^{\alpha-1}
    (t-\varphi)^k\dto \varphi.
    \end{align}
Now, as in Proposition \ref{prop:int_pol_inc_0_t}, use the substitution $z:=\fr{\varphi}{t}$, providing
\begin{align}
    \fr{t^{\alpha+k}}{\Gamma(\alpha)}
    \int_0^{1-b/t} z^{\alpha-1}(1-z)^k\dto z
    =
    \fr{t^{\alpha+k}}{\Gamma(\alpha)}
    \mathrm{B}_{1-b/t}(\alpha, k+1).
\end{align}
\end{proof}
We establish fractional differentiation using the generalized Hilfer derivative.
\begin{definition}[Hilfer fractional derivative, \cite{Hilfer2000}]
    \label{def:intro_hilfder}
    The Hilfer fractional derivative of order $0< \alpha <1$ and type $0 \ \leq \beta \leq 1$ for $t > 0 $ is given as
    \begin{align}
        \hilfder{0}{t}{\alpha, \beta}
        =
        \rlint{0}{t}{\beta(1-\alpha)}
        \der{ }{t}
        \rlint{0}{t}{(1-\beta)(1-\alpha)},
    \end{align}
    where $\rlint{0}{t}{\,}$ denotes the left-sided Riemann-Liouville fractional integral as defined in Definition \ref{def:rl_int}.
\end{definition}
\begin{remark}[Interpolation of Caputo and RL-derivative, \cite{Furati2012}]
    The Hilfer fractional derivative interpolates two of the most commonly used fractional derivative operators:
    \begin{align}
        \hilfder{0}{t}{\alpha, \beta}
        =
        \begin{cases}
            \der{}{t}
            \rlint{0}{t}{1-\alpha}  
            ,
            \quad
            \text{for}
            \quad
            \beta = 0,
            \quad
            \text{(Riemann-Liouville derivative)},
            \\
            \rlint{0}{t}{1-\alpha}
            \der{}{t}
            ,
            \quad
            \text{for}
            \quad
            \beta = 1,
            \quad
            \text{(Caputo derivative)}.
        \end{cases}
    \end{align}
\end{remark}
To obtain solutions to Hilfer fractional derivative FDEs, \citet{Furati2012} introduce the following space of weighted continuous functions
\[
C_{1-\gamma}[\astart, \aend]
=
\{  f: (\astart, \aend] \to \mathbb{R} : (t-a)^{1-\gamma}  f \in C[\astart, \aend] \}.
\]
This space is a Banach space under the corresponding norm
\begin{align}
    \| f \|_{C_{1-\gamma}[a,b]}
    =
    \sup_{t\in (a,b]}
    |(t-a)^{1-\gamma}f(t)|.
\end{align}
Moreover, the authors introduce the following fractionally differentiable subspace:
    \begin{align}
        C^{\zeta}_{1-\gamma}[\astart, \aend]
        =
        \{  f \in C_{1-\gamma}[\astart, \aend]
        , \, \hilfder{\aplus}{t}{\zeta, 0} f(t) \in C_{1-\gamma}[\astart, \aend] \}.
    \end{align}
For IVP \eqref{eq:ivp}, the following existence and uniqueness result holds:
\begin{theorem}[Existence and uniqueness \cite{Furati2012}]
\label{thm:ex_un}
    If $\bf{f}:(\astart, \aend] \times \mathbb{R}^d \to \mathbb{R}^d$ satisfies:
    \begin{enumerate}
        \item  $\bf{f}(\cdot, \bf{y}(\cdot)) \in C^{{\beta}(1-{\alpha})}_{1-{\gamma}}[\astart, \aend]$ for any $\bf{y}\in C_{1-{\gamma}}[\astart, \aend]$,
        \item $\bf{f}$ is Lipschitz in its second argument. 
        For some nonnegative matrix $K \in \mathbb{R}_{\geq 0}^{d\times d}$:
        \begin{align}
            |\bf{f}(t, \bf{y}_1(t)) - \bf{f}(t, \bf{y}_2(t))|
            \leq 
            K |\bf{y}_1(t) - \bf{y}_2(t)|,
        \end{align}
        for all $t\in (\astart, \aend]$ and $\bf{y}_1(t), \bf{y}_2(t) \in G\subset \mathbb{R}^d$. 
    \end{enumerate}
    Then, there exists a unique solution $\bf{y} \in C_{1-\gamma}^\gamma[\astart, \aend]$ to IVP \eqref{eq:ivp}.
    Furthermore, $\bf{y}$ is a solution of \eqref{eq:ivp} if and only if it satisfies the equivalent integral equation:
    \begin{align}
        \label{eq:eq_int_eq}
        \bf{y}(t)
        =
        \fr{\bf{\tilde{y}}_0}
        {\Gamma(\gamma)}(t-\astart)^{\gamma-1}
        +
        \rlint{\astart}{t}{\alpha}\bf{f}(t, \bf{y}(t)), \quad t > \astart.
    \end{align}
\end{theorem}
\begin{proof}
    The proof follows from Theorem 23 and 25 of \citet{Furati2012}, applied component-wise for vector-valued functions in $\mathbb{R}^d$.
\end{proof}
\subsection{Polynomial splines approximation}
\label{sec:prelim}
\label{sec:bs_splines}
\begin{definition}[Bernstein polynomial operator \cite{Lorentz1953}]
\label{def:bernstein_operator}
    For a function $f:[a,b]\to \mathbb{R}$, the Bernstein operator of polynomial order $q\in \mathbb{N}$ is given as:
    \begin{align}
        B^q
        f(t)
        =
        \sum_{j=0}^q
        f
        \left(
        a
        +
        {j(b-a)}/{q}
        \right)
        \binom{q}{j}
        \left(
        \fr{t-a}{b-a}
        \right)^j
        \left(
        \fr{b-t}{b-a}
        \right)^{q-j}
        .
    \end{align}
\end{definition}
\begin{remark}[Bernstein operator at endpoints]
\label{rem:bernstein_endpoints}
    For a function $f:[a,b]\to \mathbb{R}$,
    direct calculation using Definition \ref{def:bernstein_operator} gives:
    \begin{align}
        B^q f(a) = f(a),
        \quad\mathrm{ and }\quad
        B^q f(b) = f(b).
    \end{align}
\end{remark}
\begin{definition}[Modulus of continuity]
\label{def:mod_cont}
    For a continuous function $f:[a,b]\to \mathbb{R}$, the modulus of continuity $\omega$ is defined as:
    \begin{align}
        \omega\,(f;\delta)
        =
        \sup_{|t-s|\leq \delta} |f(t)-f(s)|
        ,
        \quad 
        \textrm{ for } 
        t, s \in [a,b].
    \end{align}
\end{definition}

\begin{theorem}[Bernstein operator error bound, \cite{Lorentz1953}]
\label{thm:bernstein_pol_error}
For a continuous function $f:[a,b]\to\mathbb{R}$, the following error estimate holds:
\begin{align}
    |f(t) - B^qf(t)|
    \leq
    \fr{5}{4}\,
    \omega \left(f; 
    \fr{(b-a)}{\sqrt{q}}
    \right), 
    \quad \textrm{for all }
    t\in[a,b].
\end{align}
\end{theorem}
To obtain a polynomial approximation for longer time intervals without increasing order $q$, a splines approach is taken.
We first formalize the subdivision of the time domain:
\begin{definition}[Knot collection]
    Given a closed interval $[a,b] \subset \mathbb{R}$, we define a knot collection $\mathcal{A}$ of size $k+1$ as the collection of disjoint intervals:
    \begin{align}
        \mathcal{A} = \{A_0, ..., A_k\} 
    = \{[t_0, t_1), ..., [t_{k-1}, t_{k}) ,[t_k, t_{k+1}]\},
    \end{align}
    such that 
    $
        \bigcup_{A_i\in\mathcal{A}}A_i = [a,b].
    $
\end{definition}
Now, using Definition \ref{def:bernstein_operator}, we define the Bernstein splines operator as follows:
\begin{definition}[Bernstein spline operator]
\label{def:spline_operator}
    For a function $f:[a,b] \to \mathbb{R}$, and a knot collection $\mathcal{A}$ of size $k + 1$, we define the order $q$ Bernstein spline operator as:
    \begin{align}
        \spline{q} f (t)
        =
        \sum_{i=0}^k
        \One_{A_i}
        \, 
        B^q \, \{\One_{\bar{A}_i}f\}(t),
    \end{align}
    where $\bar{A}_i$ denotes the closure of $A_i$.
\end{definition}
\begin{remark}[Continuity of the Bernstein splines operator]
    By Remark \ref{rem:bernstein_endpoints} and the continuity of $B^q$ by the continuity polynomials, we have that $\spline{q} f$ is continuous for continuous $f$.
\end{remark}
We can now give an error bound in approximation using the splines operator:

\begin{theorem}[Bernstein splines error bound]
\label{thm:bernstein_spline_error}
    Let $f:[a,b]\to \mathbb{R}$ be a continuous function with modulus of continuity $\omega$.
    Let $h = \sup_{A_i\in \mathcal{A}}|A_i|$ be the maximum knot size of knot collection $\mathcal{A}$ on $[a,b]$.
    Then,
    \begin{align}
        |f(t) - \spline{q} f(t)|
        &\leq
        \fr{5}{4}\,\omega
        \left(
        f;\fr{h}{\sqrt{q}}
        \right),
        \quad 
        \text{for all } t \in [a,b].
    \end{align}
\end{theorem}
\begin{proof}
    The proof follows by combining the operator definition with the polynomial error bound of Theorem \ref{thm:bernstein_pol_error}.
    \end{proof}
\begin{definition}[Spline fractional integration operator]
\label{def:splint_operator}
For $u: [a,b] \to \mathbb{R}$ continuous and knot collection $\mathcal{A}$, the order $0<\alpha<1$ splines fractional integration operator $\splint{a}{t}{\alpha}$ is given as:
\begin{align}
    \splint{a}{t}{\alpha} u(t)
    =
    \spline{q}
    \, \rlint{a}{t}{\alpha}
    u (t), \quad \textrm{ for } a \leq t \leq b,
\end{align}
where $\spline{q}$ is the spline operator as defined in Definition \ref{def:spline_operator}.
\end{definition}
To provide error bounds, we can use the results as obtained above:
\begin{lemma}[Error of the spline integration operator] \label{lemma:spline_int_error}
Let $u:[a,b]\to \mathbb{R}$ be a continuous function.
Then,
\begin{align}
    \Cnorm{a}{b}{\rlint{a}{t}{\alpha} u - \splint{a}{t}{\alpha} u}
    \leq
    \left( \fr{h}{\sqrt{q}} \right)^\alpha
    \fr{5\,\Cnorm{a}{b}{u}
    }{2\,\Gamma(\alpha+1) },
\end{align}
where $\Cnorm{a}{b}{f} := \sup_{t\in (a, b]} | f(t)|$ and $h = \sup_{A_i\in \mathcal{A}}|A_i|$ is the maximum knot size of knot collection $\mathcal{A}$ on $[a,b]$.
\end{lemma}
\begin{proof}
By the Definition \ref{def:splint_operator} and Theorem \ref{thm:bernstein_spline_error} we have that for $a\leq t\leq b$
    \begin{align}
        |\rlint{a}{t}{\alpha} u(t) - \splint{a}{t}{\alpha} u(t)|
        =
        \left|
        \rlint{a}{t}{\alpha} \,
        u(t)
        -
        \spline{q}
        \, \rlint{a}{t}{\alpha} u(t)
        \right|
        \leq
        \fr{5}{4}\,\omega \left(
        \rlint{a}{t}{\alpha}u(t); \fr{h}{\sqrt{q}}
        \right).
    \end{align}
Now, let us bound the modulus of continuity.
Take $a \leq t < t'\leq b$. Then:
    \begin{gather}
        |
        \rlint{a}{t'}{\alpha} u(t')
        -
        \rlint{a}{t}{\alpha} u(t)
        |
        =
        \fr{1}{\Gamma(\alpha)}
        \bigg{|}
        \int_a^{t'} u(s)(t'-s)^{\alpha-1}\dto s
        -
        \int_a^{t} u(s)(t-s)^{\alpha-1}\dto s
        \bigg{|}
        \\
        =
        \fr{1}{\Gamma(\alpha)}
        \bigg{|}
        \int_a^{t} u(s)
        \left[
        (t'-s)^{\alpha-1}
        -
        (t-s)^{\alpha-1}
        \right]\dto s 
        +
        \int_t^{t'} u(s)(t'-s)^{\alpha-1}\dto s
        \bigg{|}.
    \end{gather}
Since $\alpha-1 \in (-1, 0)$ and $(t'-s) > (t-s) \geq 0$, we have $(t-s)^{\alpha-1} - (t'-s)^{\alpha-1} \geq 0$ for $a\leq s\leq t$.
Furthermore, $(t'-s)^{\alpha-1} \geq 0$ for $t \leq s \leq t'$.
Hence, we can apply the triangle inequality and Hölder's inequality twice, yielding:
    \begin{gather}
        |
        \rlint{a}{t'}{\alpha} u(t')
        -
        \rlint{a}{t}{\alpha} u(t)
        |
        \\
        \leq
        \fr{1}{\Gamma(\alpha)}
        \left(
        \int_a^{t}
        \left(
        (t-s)^{\alpha-1}
        -
        (t'-s)^{\alpha-1}
        \right)
        \mathrm{d}s
        \, 
        +
        \int_t^{t'}(t'-s)^{\alpha-1}
        \mathrm{d}s
        \, 
        \right)\Cnorm{a}{b}{u}
        \\
        =
        \fr{
        -
        \left(
        (t'-a)^\alpha
        -(t-a)^\alpha
        \right)
        +2(t'-t)^\alpha
        }{\Gamma(\alpha+1)}
        \Cnorm{a}{b}{u}
        \leq 
        \fr{2(t'-t)^\alpha}{\Gamma(\alpha+1)}
        \Cnorm{a}{b}{u}.
        \label{eq:app_rl_cont_t_bounds}
    \end{gather}
    Here, the final estimate follows since $t'>t$ gives $(t'-a)^\alpha-(t-a)^\alpha > 0$.
Finally, taking the maximum distance $t' - t = h$, we have
\begin{align}
    \omega \left(
        \rlint{a}{t}{\alpha}u(t); \fr{h}{\sqrt{q}}
        \right)
    =
    \fr{2(hq^{-1/2})^\alpha}{\Gamma(\alpha+1)}
    \Cnorm{a}{b}{u},
\end{align}
resulting in the error bound and concluding the proof.
\end{proof}
\begin{lemma}[Bound of the spline integration operator]
\label{thm:splint_bound}
    Take a continuous function $u: [a,b]\to \mathbb{R}$.
    Then, for $a\leq t\leq b$, the operator $\splint{a}{t}{\alpha}$ is bounded. In particular, 
    \begin{align}
        \Cnorm{a}{b}{\splint{a}{t}{\alpha} u}
        \leq
        \Psi
        \Cnorm{a}{b}{u},
        \quad 
        \text{with}
        \quad
        \Psi=\fr{
        \fr{5}{2}
        \left(
        \fr{h}{\sqrt{q}}
        \right)^\alpha
        + (t-a)^\alpha
        }{\Gamma(\alpha+1)}.
    \end{align}
    Here $h = \sup_{A_i \in \mathcal{A}}|A_i|$  is the maximum knot size and $q$ is the order of Bernstein basis polynomial splines.
    Furthermore, since $\splint{a}{t}{\alpha}$ is linear, this also gives continuity in $u$.
\end{lemma}
\begin{proof}
First, write:
\begin{align}
        \Cnorm{a}{b}{\splint{a}{t}{\alpha} u}
        \leq
        \Cnorm{a}{b}{\splint{a}{t}{\alpha}u
        -
        \rlint{a}{t}{\alpha}u
        +
        \rlint{a}{t}{\alpha}u
        }
        \leq
        \Cnorm{a}{b}{\splint{a}{t}{\alpha}u
        -
        \rlint{a}{t}{\alpha}u
        }
        +
        \Cnorm{a}{b}{
        \rlint{a}{t}{\alpha}u
        }.
    \end{align}
Then, by the bounds of Lemma \ref{lemma:spline_int_error} and Proposition \ref{prop:rl_bdd} for $a\leq t$ we have:
\begin{align}
    \Cnorm{a}{b}{\splint{a}{t}{\alpha} u}
    \leq
    \fr{
    \fr{5}{2}
    \left(
    \fr{h}{\sqrt{q}}
    \right)^\alpha
    + (t-a)^\alpha
    }{\Gamma(\alpha+1)}
    \Cnorm{a}{b}{
    u
    },
\end{align}
providing the operator bound.
\end{proof}
\section{Bernstein spline approximation results}
\label{sec:spline_conv}
We first establish some notational conventions for this and the following sections.

For a matrix $A \in \mathbb{R}^{m \times n}$, we denote the supremum-induced matrix norm as
\begin{align}
    \|A\|_{\infty }=\max _{1\leq i\leq m}\sum _{j=1}^{n}\left|A_{ij}\right|.
\end{align}

We use the following notation for the weighted function space $C_{1-\gamma}[0, T]$ and ``$\eps-$shifted'' weighted space $C_{1-\gamma}[\eps, T]$ respectively:
\begin{gather}
    C_{1-\gamma}[0, T]
    =
    \{  f: (0, T] \to \mathbb{R}^d : t^{1-\gamma} f \in C[0,T] \},
    \\
    \weighnorm{f}
    =
    \sup_{t\in(0,T]}
    |
    t^{1-\gamma}
    f(t)
    |,
    \\
    C_{1-\gamma}[\eps, T]
    =
    \{  f: (\eps, T] \to \mathbb{R}^d : t^{1-\gamma} f \in C[\eps,T] \},
    \label{eq:shifted_norms}
    \\
    \epsweighnorm{f}
    =
    \sup_{t\in(\eps,T]}
    |
    t^{1-\gamma}
    f(t)
    |.
\end{gather}
\begin{proposition}
\label{prop:eps_norms_eq}
    $\eps$-shifted weighted continuous functions are continuous,
    i.e., for $0<\eps<t\leq T$,
    \begin{align}
        C_{1-\gamma}[\eps, T]
        =
        C[\eps, T].
    \end{align}
    \begin{proof}
        First, assume $f \in C_{1-\gamma}[\eps, T]$. 
        Hence, by \eqref{eq:shifted_norms} we know $t^{1-\gamma} f \in C[\eps, T]$.
        Then,  since $t\in (\eps, T]$ with $\eps>0$, $t^{\gamma-1}$ is bounded and continuous.
        Hence by continuity of composition by multiplication, $t^{\gamma-1} t^{1-\gamma}f = f \in C[\eps, T]$.
        For the reverse implication, assume $f\in C[\eps, T]$.
        Then, by the same argument, $t^{1-\gamma}f \in C[\eps, T]$ and thus $f\in C_{1-\gamma}[\eps, T]$, proving the claim.
    \end{proof}
\end{proposition}
We now present the main splines approximation theorem:
\begin{theorem}[Hilfer derivative spline approximations]
\label{thm:hilf_splines}
    Consider FIVP \eqref{eq:ivp}.
    Assume the requirements of existence and uniqueness of Theorem \ref{thm:ex_un} hold.
    For a knot collection $\mathcal{A} = \{A_0, ..., A_k\} $ on the interval $[\eps, T]$ with $\eps>0$, we can construct an order $q\in \mathbb{N}$ Bernstein spline approximation $\bf{y}^{q,\eps}_{j, \bf{n}} : [\eps, t_{j+1}] \to \mathbb{R}^d$ for $j \in \{0, \ldots k\}$ and $\bf{n} \in \mathbb{N}^j$, such that if 
    \begin{align}
        \left\|
        \left(1+\fr{h_i}{t_i}\right)^{1-\gamma} \, \Psi_i \, K
        \right\|_\infty
        < 1,
        \quad
        \text{ for all }i\in \{0, ..., k\},
    \end{align}
    $\text{where }\Psi_i=\fr{
        \fr{5}{2}
        \left(
        \fr{h_i}{\sqrt{q}}
        \right)^\alpha
        + h_i^\alpha
        }{\Gamma(\alpha+1)} \text{ and } \, h_i = |t_{i+1}-t_i|,$
    then $\bf{y}^{q,\eps} := \Lim{\bf{n}\to\infty} \bf{y}_{k, \bf{n}}^{q,\eps}$ satisfies
    \begin{align}
        \bf{y}^{q,\eps}(t)
        =
        \fr{\bf{\tilde{y}}_0}{\Gamma(\gamma)}
        t^{\gamma-1}
        +
        \splint{\eps}{t}{\alpha}\bf{f}(t, \bf{y}^{q,\eps}(t)), 
        \quad \eps \leq  t \leq T.
        \label{eq:y_q_eps_sol}
    \end{align}
    Furthermore, $\bf{y}^{q, \eps} \in C[\eps, T]$ and $\bf{y}^{q,\eps}_{j, \bf{n}} \in C[\eps, t_{j+1}]$ for all $j \in \{0, \ldots k\}$ and $\bf{n} \in \mathbb{N}^j$.
\end{theorem}
\begin{proof}
    Consider the equivalent integral equation \eqref{eq:eq_int_eq} of IVP \eqref{eq:ivp}.
    Taking the transformation $\bf{v}(t) := t^{1-\gamma} \bf{y}(t)$ then gives:
    \begin{align}
        \bf{v}(t)
        =
        \bf{\tilde{v}_0}
        +
        \rlint{0}{t}{\alpha}
        \bf{g}(t,s,\bf{v}(s))
        ,
        \quad
        t> 0.
    \end{align}
    Here, $\bf{\tilde{v}}_0 := \fr{\bf{\tilde{y}}_0}{\Gamma(\gamma)}$ and $\bf{g}(t,s,\bf{v}(s)) := t^{1-\gamma}
    \bf{f}(s, s^{\gamma-1} \bf{v}(s))$, where $s\leq t$ is introduced as the variable of integration for notational clarity.

    We introduce the corresponding $\eps$-shifted integral equation:
    \begin{align}
        \bf{v}^\eps(t)
        =
        \bf{\tilde{v}_0}
        +
        \rlint{\eps}{t}{\alpha}
        \bf{g}(t,s,\bf{v}^\eps(s))
        ,
        \quad
        t\geq \eps
        .\label{eq:hilf_ivp_eps_shifted}
    \end{align}
    Since by assumption $\bf{f}(\cdot, \bf{y}(\cdot)) \in C_{1-\gamma}[0, T]$ for any $\bf{y} \in C_{1-\gamma}[0,T]$, we also have $\bf{f}(\cdot, \bf{y}(\cdot)) \in C_{1-\gamma}[\eps, T]$ and hence $\bf{f}(\cdot, \bf{y}(\cdot)) \in C[\eps, T]$ by Proposition~\ref{prop:eps_norms_eq}.
    Thus, we observe that by construction, $\bf{g}(t, \cdot, \bf{v}^\eps(\cdot)) \in C[\eps, T]$ for any $\bf{v}^\eps \in C[\eps, T]$.
    Furthermore, given  $\eps \leq s\leq t$, we have that $\bf{g}$ is Lipschitz in $\bf{v}^\eps$ by Lipschitz continuity of $\bf{f}$.
    Hence, we have satisfied the requirements of Theorem \ref{thm:ex_un} for $\beta = 1$, $a = \eps$, $b = T$ and thus know that a unique solution $\bf{v}^\eps \in C[\eps, T]$  exists for integral equation~\eqref{eq:hilf_ivp_eps_shifted}.

    To obtain approximations to solutions $\bf{v}^\eps$ of \eqref{eq:hilf_ivp_eps_shifted}, we construct the solution spline approximation sequence $\bf{v}^{q,\eps}_{j, \bf{n}} : [\eps,t_{j+1}]\to\mathbb{R}^d$ as follows:
    \begin{align}
    \label{eq:local_splines_ind_build}
        \bf{v}^{q,\eps}_{j, \bf{n}}(t) 
        =
        \sum_{i=0}^j \bf{u}^{q,\eps}_{i, n_i} \One_{A_i}(t)
        =
        \begin{cases}
            \bf{u}^{q,\eps}_{0, n_0}(t), \quad \textrm{if} \quad &t \in A_0,\\
            \bf{u}^{q,\eps}_{1, n_1}(t), \quad \textrm{if} \quad &t \in A_1, \\
            \quad \vdots\\
            \bf{u}^{q,\eps}_{j, n_j}(t), \quad \textrm{if} \quad &t \in A_j. \\
        \end{cases}
    \end{align}
    Here, $\bf{u}_{i, n} : A_i \to \mathbb{R}^d$ satisfies:
    \begin{align}
    \label{eq:local_splines_it}
        (\bf{u}^{q,\eps}_{i,{n+1}}(t))_{n\geq 0}
        =
        \bf{\tilde{v}}_0
        +
        \splint{\eps}{t}{\alpha}
        \bf{g}(t, s,\bf{v}^{q, \eps}_{i, \bf{n}}(s)),
        \quad \bf{n}_i = n,
        \quad t \in A_i,
    \end{align}
    where $\splint{\eps}{t}{\alpha}$ is the splines integration operator as defined in Definition \ref{def:splint_operator}.
    For initial iteration values, we prescribe:
    \begin{align}
        \begin{cases}
        \label{eq:splines_initial}
        \bf{u}^{q,\eps}_{i, 0}
        :=
        {\bf{\tilde{v}}}_0,
        \quad &\textrm{for}\quad  i =0,
        \\
        \bf{u}^{q,\eps}_{i, 0}
        :=
        \Lim{t \uparrow t_i}
        \bf{u}^{q,\eps}_{i-1, \bf{n}_{i-1}}
        (t),
        \quad &\textrm{for}\quad  i \geq 1.
        \end{cases}
    \end{align}
    Then, the limit of the splines approximation sequence on the whole interval $[\eps,T]$ is given by:
    \begin{align}
        \label{eq:y_q_limit_form}
        \bf{v}^{q,\eps}(t)
        :=
        \bf{v}^{q,\eps}_{k, \infty} (t)
        :=
        \sum_{i=0}^k
        \lim_{n\to\infty}
        \bf{u}^q_{i,n}
        \One_{A_i}(t).
    \end{align}
    To prove convergence of \eqref{eq:y_q_limit_form}, we first need to show that all intermediate iterations $\bf{v}^{q,\eps}_{j, \bf{n}}$ are continuous.
    Take $q \in \mathbb{N}$.
    Then, for $j = 0$, we know that $\bf{v}^{q,\eps}_{0, {n}_0}$ is continuous for all ${n}_0 \in \mathbb{N}$ since $\bf{v}^{q,\eps}_{0, 0}$ is continuous by \eqref{eq:splines_initial} as it is a constant function, and for $ {n}_0 \geq 1$, continuity is obtained by the (Lipschitz) continuity of $\bf{g}$ and by Lemma \ref{thm:splint_bound}.
    Then, for $\bf{v}^{q,\eps}_{1, \bf{n}}$, the initial iteration is again continuous by \eqref{eq:splines_initial} since $\bf{u}^{q,\eps}_{1, 0} (t_1)
        =
        \Lim{t \uparrow t_1}
        \bf{u}^{q,\eps}_{0, \bf{n}_{0}}
        (t)$.
    Repeating the same argument as above, we obtain continuity of $\bf{v}^{q,\eps}_{1,\bf{n}}$ for $n_1 \geq 1$, where we note that continuity at $t=t_1$ is conserved since $\Lim{t\uparrow t_1} \splint{\eps}{t}{\alpha} = \splint{\eps}{t_1}{\alpha} $ by Lemma \ref{lemma:spline_int_error}.
    Finally, repeating the same argumentation for $j \geq 2$, we obtain continuity for all $\bf{v}^{q,\eps}_{j, \bf{n}}$.
    
    Using the bound of Lemma  \ref{thm:splint_bound}, we can now show the convergence of sequence \eqref{eq:local_splines_it} inside the summation of \eqref{eq:y_q_limit_form}:
    \begin{align}
        &\supnorm{t}{A_i}{
        \bf{u}^{q,\eps}_{i, n+1}(t)
        -
        \bf{u}^{q,\eps}_{i, n}(t)
        }
        \\
        &=
        \supnorm{t}{A_i}{
        \splint{\eps}{t}{\alpha}
        \{
        \bf{g}(t,s, \bf{v}^{q,\eps}_{i, \bf{n}}(t))
        -
        \bf{g}(t,s, \bf{v}^{q,\eps}_{i, \bf{n}-1}(t))
        \}
        }
        \\
        &=
        \supnorm{t}{A_i}{
        \splint{{t_i}}{t}{\alpha}
        \{
        \bf{g}(t,s, \bf{u}^{q,\eps}_{i, {n}}(t))
        -
        \bf{g}(t,s, \bf{u}^{q,\eps}_{i, {n}-1}(t))
        \}
        }
        \\
        &\leq
        \Psi_i \supnorm{s,t}{A_i}{
        \bf{g}(t,s, \bf{u}^{q,\eps}_{i, {n}}(s))
        -
        \bf{g}(t,s, \bf{u}^{q,\eps}_{i, {n}-1}(s))
        }
        \\
        &=
        \Psi_i \supnorm{s,t}{A_i}{
        t^{1-\gamma}
        \bf{f}(s, s^{\gamma-1} \bf{u}^{q,\eps}_{i, {n}}(s))
        -
        t^{1-\gamma}
        \bf{f}(s, s^{\gamma-1} \bf{u}^{q,\eps}_{i, {n}-1}(s))}
        \\
        &\leq
        \Psi_i 
        {t}_{i+1}^{1-\gamma}
        \supnorm{s}{A_i}{
        \bf{f}(s, s^{\gamma-1} \bf{u}^{q,\eps}_{i, {n}}(s))
        -
        \bf{f}(s, s^{\gamma-1} \bf{u}^{q,\eps}_{i, {n}-1}(s))}
        \\
        &\leq
        \Psi_i 
        {t}_{i+1}^{1-\gamma}
        K
        \supnorm{s}{A_i}{
         s^{\gamma-1}
        \bf{u}^{q,\eps}_{i, {n}}(s)
        -
        s^{\gamma-1}
        \bf{u}^{q,\eps}_{i, {n}-1}(s)
        }
        \\
        &\leq
        \left(1+\fr{{h}_{i}}{t_i}\right)^{1-\gamma}
        \Psi_i 
        K
        \supnorm{s}{A_i}{
        \bf{u}^{q,\eps}_{i, {n}}(s)
        -
        \bf{u}^{q,\eps}_{i, {n}-1}(s)
        }.
    \end{align}
    Since by assumption $\| (1+\fr{{h}_{i}}{t_i})^{1-\gamma}
        \Psi_i 
        K
    \|_\infty < 1$, Banach's fixed point theorem gives uniform convergence.
    Hence, given $i \in \{0, ..., k\}$, we have convergence for the sequence $(\bf{u}^{q,\eps}_{i, n})_{n\geq 0}$ as $n\to \infty$.
    From \eqref{eq:local_splines_ind_build} and \eqref{eq:local_splines_it} it now follows that $\bf{v}^{q,\eps} : [\eps,T] \to \mathbb{R}^d$ satisfies 
    \begin{align}
        \bf{v}^{q, \eps}(t)
        =
        \bf{\tilde{v}_0}
        +
        \splint{\eps}{t}{\alpha}
        \bf{g}(t,s,\bf{v}^{q,\eps}(s))
        ,
        \quad
        \eps \leq t \leq T
        ,
    \end{align}
    and by transforming back to $\bf{y}^{q,\eps}(t):= t^{\gamma-1}\bf{v}^{q,\eps} (t)$ we thus obtain \eqref{eq:y_q_eps_sol}.
    Finally, we note that continuity of $\bf{v}^{q, \eps}$ is obtained by the uniform convergence of $(\bf{u}^{q, \eps}_{i,n})_{n\geq 0}$ and the continuity at connecting spline points as shown above, providing continuity of $\bf{y}^{q,\eps}$ after transforming back, since $\eps \leq t \leq T$.
\end{proof}
We now present the main convergence results and rates with respect to analytical solutions of the IVP:
\begin{theorem}[Convergence rates]
\label{thm:convergence_rates}
    The approximations $\bf{y}^{q,\eps}$ of Theorem \ref{thm:hilf_splines} satisfy:
    \renewcommand{\theenumi}{(\roman{enumi})}%
    \begin{enumerate}
    \item
    \label{claim:conv_pol}
    \textbf{Convergence in polynomial order.} $\lim_{q\to\infty}\bf{y}^{q, \eps} = \bf{y}^\eps $ with order
        \begin{align}
            \epsweighnorm{
            \bf{y}^{q,\eps}
            -
            \bf{y}^\eps
            }
            =
            \mathcal{O}({q}^{-
            {\alpha}/{2}
            }).
        \end{align}
        \item \label{claim:conv_h} \textbf{Convergence in knot size.} $ \lim_{h\to 0}\bf{y}^{q, \eps} = \bf{y}^\eps$
         with order
        \begin{align}
            \epsweighnorm{
            \bf{y}^{q,\eps}
            -
            \bf{y}^\eps
            }
            =
            \mathcal{O}(h^\alpha),
        \end{align}
        where $h := \max_{i \in \{0, \ldots, k\}}h_i$ is the maximum knot size of $\mathcal{A}$.
    \end{enumerate}
    Here, $\bf{y}^\eps : [\eps, T] \to \Rbb^d$ is the solution to the shifted analytical integral equation:
    \begin{align}
        \label{eq:y_eps_sol}
        \bf{y}^{\eps}(t)
        =
        \fr{\bf{\tilde{y}_0}}{\Gamma(\gamma)}
        t^{\gamma-1}
        +
        \rlint{\eps}{t}{\alpha}
        \bf{f}(s,
        \bf{y}^{q,\eps}(s)
        )
        ,
        \quad
        \eps \leq t \leq T
        .
    \end{align}
    Furthermore, we have:
    \renewcommand{\theenumi}{(\roman{enumi})}%
    \begin{enumerate}
    \setcounter{enumi}{2}
        \item 
        \label{claim:conv_eps}
        \textbf{{Convergence to the analytical solution.}}
        $\lim_{\eps \downarrow 0} \bf{y}^\eps = \bf{y}$ with order
        \begin{align}
            \epsweighnorm{\bf{y}^\eps -
            \bf{y}}
            = \mathcal{O}(\eps^{\alpha}).
        \end{align}
    \end{enumerate}
\end{theorem}
\begin{proof}
    Consider $\bf{y}^{q,\eps}$ and $\bf{y}^\eps$ such that
    \eqref{eq:y_q_eps_sol}  and \eqref{eq:y_eps_sol} hold respectively.
    Then,
    \begin{align}
        \bf{y}^{q,\eps}(t)
        -
        \bf{y}^\eps(t)
        = 
        (\splint{\eps}{t}{\alpha} 
        - \rlint{\eps}{t}{\alpha}) \bf{f}(t, \bf{y}^{q,\eps}(t))
        + \rlint{\eps}{t}{\alpha} \{ 
        \bf{f}(t, \bf{y}^{q,\eps}(t))
        -
        \bf{f}(t, \bf{y}^\eps(t))\}.
    \end{align}
    And thus, by Lemma \ref{lemma:spline_int_error}:
    \begin{align}
        \label{eq:local_splines_before_gronwall}
        \overline{
        |
        \bf{y}^{q,\eps}(t)
        -
        \bf{y}^\eps(t)
        |}
        \leq
        \left( \fr{h}{\sqrt{q}} \right)^\alpha
        \fr{5\,
        \overline{\|\bf{f} \|}_{C[0,T]_{1-\gamma}}
        }{2\,\Gamma(\alpha+1) }
        +
        ||K||_\infty\,
        \rlint{\eps}{t}{\alpha}
        \overline{|\bf{y}^{q,\eps}(t)-\bf{y}^\eps(t)|}
        .
    \end{align}
    Here, $\overline{\bf{v}} := \sup_{i \in \{1, ..., d\}} v_i$ denotes the supremum over all components of the $d$-dimensional vector $\bf{v}$, making \eqref{eq:local_splines_before_gronwall} a scalar equation.
    Since $\left( \fr{h}{\sqrt{q}} \right)^\alpha
        \fr{5\,
        \overline{\|\bf{f} \|}_{C_{1-\gamma}[0,T]}
    }{2\,\Gamma(\alpha+1) }$ does not depend on $t$, we can apply the fractional Grönwall's inequality (Corollary 2, \cite{Ye2007}) to obtain
    \begin{align}
    \label{eq:local_splines_gronwall}
     \overline{
        |
        \bf{y}^{q,\eps}(t)
        -
        \bf{y}^\eps(t)
        |}
        \leq
        \left( \fr{h}{\sqrt{q}} \right)^\alpha
        \fr{5\,\overline{\| \bf{f} \|}_{C_{1-\gamma}[0,T]}
        }{2\,\Gamma(\alpha+1) }
        \mathrm{E}_\alpha(t^\alpha ||K||_\infty),
    \end{align}
    where $\mathrm{E}_\alpha(z) = \sum_{j=0}^\infty \fr{z^j}{\Gamma(j\alpha+1)}$ denotes the Mittag-Leffler function.
    Then, since $\mathrm{E}_\alpha(t^\alpha ||K||_\infty)$ is independent of $q$ and $h$, and since the supremum is taken over all dimensions, we obtain convergence orders of
    \begin{align}
        \sup_{t\in [0,T]}
        \epsweighnorm{
        \bf{y}^{q,\eps}
        -
        \bf{y}^\eps
        }
        =
        \mathcal{O}(h^\alpha)
        ,
        \quad 
        \text{and}
        \quad
        \epsweighnorm{
        \bf{y}^{q,\eps}
        -
        \bf{y}^\eps
        }
        =
        \mathcal{O}({q}^{-
        {\alpha}/{2}
        }),
    \end{align}
    proving claim \ref{claim:conv_pol} and \ref{claim:conv_h}.
    
    Finally, we prove convergence to the analytical solution as in claim \ref{claim:conv_eps}.
    Consider analytical solution $\bf{y} : [0,T] \to \Rbb^d$ satisfying the equivalent integral equation \eqref{eq:eq_int_eq} of FIVP \eqref{eq:ivp}.
    Then, for $\eps\leq t \leq T$:
    \begin{align}
        t^{1-\gamma}
        \left[
        \bf{y}(t) - \bf{y}^\eps(t)
        \right]
        &=
        t^{1-\gamma}
        \rlint{0}{t}{\alpha}
        \bf{f}(t, \bf{y}(t))
        -
        t^{1-\gamma}
        \rlint{\eps}{t}{\alpha}
        \bf{f}(t, \bf{y}^\eps(t))
        \\
        &=
        t^{1-\gamma}
        \rlint{0}{\eps}{\alpha}
        \bf{f}(t, \bf{y}(t))
        +
        t^{1-\gamma}
        \rlint{\eps}{t}{\alpha}
        [
        \bf{f}(t, \bf{y}(t))
        -
        \bf{f}(t, \bf{y}^\eps(t))
        ].
    \end{align}
    By the Lipschitz condition on $\bf{f}$ and the triangle inequality, we have:
    \begin{align}
        \label{eq:hilf_shift_main_ineq}
        t^{1-\gamma}
        |\bf{y}(t) - \bf{y}^\eps(t)|
        \leq
        t^{1-\gamma}
        |\rlint{0}{\eps}{\alpha}
        \bf{f}(t, \bf{y}(t))|
        +
        t^{1-\gamma}
        \rlint{\eps}{t}{\alpha}
        K
        |
        \bf{y}(t)
        -
        \bf{y}^\eps(t)
        |.
    \end{align}
    By making use of the H$\ddot{\text{o}}$lder inequality,
    the first term in \eqref{eq:hilf_shift_main_ineq} can be written as follows:
    \begin{align}
        t^{1-\gamma}
        |\rlint{0}{\eps}{\alpha}
        \bf{f}(t, \bf{y}(t))|
        &\leq
        \fr{t^{1-\gamma}}{\Gamma(\alpha)}
        \int_0^\eps 
        |(t-s)^{\alpha-1}\bf{f}(t, \bf{y}(t)) |\dto s
        \\
        &=
        \fr{t^{1-\gamma}}{\Gamma(\alpha)}
        \int_0^\eps 
        |(t-s)^{\alpha-1}s^{\gamma-1}\bf{f}(t, \bf{y}(t))s^{1-\gamma} |\dto s
        \\
        &\leq
        \fr{t^{1-\gamma}}{\Gamma(\alpha)}
        \int_0^\eps 
        (t-s)^{\alpha-1}s^{\gamma-1}
        \dto s \,
        \overline{\|
        \bf{f}
        \|}_{C_{1-\gamma}[0,T]}.
    \end{align}
    By Proposition \ref{prop:int_pol_inc_0_t}, we have 
    $
        \fr{t^{1-\gamma}}{\Gamma(\alpha)}
        \int_0^\eps 
        (t-s)^{\alpha-1}s^{\gamma-1}
        \dto s 
        =
        \fr{t^{\alpha}}{\Gamma(\alpha)}
        \mathrm{B}_{\eps/t}(\gamma, \alpha).
    $
    Hence, 
    \begin{align}
        t^{1-\gamma}
        |\rlint{0}{\eps}{\alpha}
        \bf{f}(t, \bf{y}(t))|
        \leq
        \fr{t^{\alpha}}{\Gamma(\alpha)}
        \mathrm{B}_{\eps/t}(\gamma, \alpha)
        \| \bf{f} \|_{C_{1-\gamma}[0,T]}
        \leq
        \underbrace{
        \sup_{t\in [\eps, T]}
        \fr{t^{\alpha}}{\Gamma(\alpha)}
        \mathrm{B}_{\eps/t}(\gamma, \alpha)
        \overline{\|
        \bf{f}
        \|}_{C_{1-\gamma}[0,T]}
        }_{=:\Theta_\eps}.
    \end{align}
    Now, denoting
    $ {e} (t):=
    \overline{t^{1-\gamma} | {y}_i(t) - {y}_i^\eps(t)|}$, we have
    \begin{align}
        t^{1-\gamma}
        \rlint{\eps}{t}{\alpha}
        K
        |
        \bf{y}(t)
        -
        \bf{y}^\eps(t)
        |
        &\leq
        \fr{1}{\Gamma(\alpha)}
        \int_\eps^t
        t^{1-\gamma}
        (t-s)^{\alpha-1}
        \|K\|_\infty
        s^{\gamma-1} {e} (s)
        \dto s.
    \end{align}
    Substituting these bounds in \eqref{eq:hilf_shift_main_ineq} then yields:
    \begin{align}
         {e} (t)
        \leq 
        \Theta_\eps
        +
        \int_\eps^t
        \fr{\|K\|_\infty}{\Gamma(\alpha)}
        \left(\fr{t}{s}\right)^{1-\gamma}
        (t-s)^{\alpha-1}
         {e} (s)
        \dto s.
    \end{align}
    Since $\Theta_\eps$ is nondecreasing, we can apply Grönwall's inequality to obtain:
    \begin{align}
    \label{eq:gronwall_result_part}
         {e} (t)
        \leq
        \Theta_\eps
        \mathrm{exp}\left\{
        \int_\eps^t
        \fr{\|K\|_\infty}{\Gamma(\alpha)}
        \left(\fr{t}{s}\right)^{1-\gamma}
        (t-s)^{\alpha-1}
        \dto s
        \right\}
        \leq
        \Theta_\eps
        \mathrm{exp}\left\{
        \fr{\|K\|_\infty 
        \mathrm{B}(\gamma, \alpha)T^{\alpha}
        }{\Gamma(\alpha)}
        \right\},
    \end{align}
    where the last expression follows from Proposition \ref{prop:int_pol_inc_t_T}.
    Now, we are left to prove
    $\lim_{\eps \downarrow 0} \Theta_\eps = 0$.
    First, observe
    \begin{align}
        \lim_{\eps \downarrow 0} \Theta_\eps 
        &=
        \fr{\overline{\|
        \bf{f}
        \|}_{C_{1-\gamma}[0,T]}}{\Gamma(\alpha)}
        \lim_{\eps \downarrow 0} 
        \sup_{t\in [\eps, T]}
        {t^{\alpha}}
        \int_0^{\eps/t}
        u^{\gamma-1}
        (1-u)^{\alpha-1}
        \dto u.
    \end{align}
    Using the substitution $s : = t/\eps u$, we have
    \begin{align}
        \int_0^{\eps/t}
        u^{\gamma-1}
        (1-u)^{\alpha-1}
        \dto u
        &=
        \left(\fr{\eps}{t}\right)^\gamma
        \int_0^{1}
        s^{\gamma-1}
        (1-s \eps / t)^{\alpha-1}
        \dto s
        \\
        &\leq 
        \left(\fr{\eps}{t}\right)^\gamma
        \int_0^{1}
        s^{\gamma-1}
        (1-s )^{\alpha-1}
        \dto s
        =
        \left(\fr{\eps}{t}\right)^\gamma
        \mathrm{B}(\gamma, \alpha),
    \end{align}
    where the last inequality follows from 
    the fact that $(1-s \eps / t)^{\alpha-1}$ is decreasing in $t \in [\eps, T]$.
    Hence, combined we have 
    \begin{align}
        \lim_{\eps \downarrow 0} \Theta_\eps 
        &\leq
        \fr{\overline{\|
        \bf{f}
        \|}_{C_{1-\gamma}[0,T]}}{\Gamma(\alpha)}
        \lim_{\eps \downarrow 0} 
        \sup_{t\in [\eps, T]}
        {t^{\alpha}}
        \left(\fr{\eps}{t}\right)^\gamma
        \mathrm{B}(\gamma, \alpha).
    \end{align}
    Since $\alpha - \gamma \leq 0$, we have $
    \sup_{t\in [\eps, T]}
        {t^{\alpha}}
        \left(\fr{\eps}{t}\right)^\gamma
        =
        \sup_{t\in [\eps, T]}
        {t^{\alpha-\gamma}}
        \eps^\gamma
        = \eps^\alpha
        $.
    Hence, $\lim_{\eps \downarrow 0 } \Theta_\eps = 0$ with $\mathcal{O}(\eps^\alpha)$.
    Combined with inequality \eqref{eq:gronwall_result_part}, we thus have $\lim_{\downarrow \eps} \|{e}\|_{C[\eps, T]} = 0$. 
    Hence, $\lim_{\downarrow \eps} \epsweighnorm{\bf{y}^\eps - \bf{y}} = \bf{0}$, proving convergence to the analytical solution of claim \ref{claim:conv_eps}.
\end{proof}
\begin{remark}[Convergence to the analytical solution]
    Under the assumptions of Theorem \ref{thm:ex_un} and \ref{thm:hilf_splines}, we thus have by Theorem \ref{thm:convergence_rates} that
    \begin{gather}
        \lim_{\eps \downarrow 0} 
        \lim_{\fr{h}{q} \downarrow 0} 
        \epsweighnorm{\bf{y}^{q,\eps} - \bf{y}}
        \\
        \leq
        \lim_{\eps \downarrow 0} 
        \lim_{\fr{h}{q} \downarrow 0} 
        \left(
        \epsweighnorm{\bf{y}^{q,\eps} - \bf{y}^\eps}
        +
        \epsweighnorm{\bf{y}^{\eps} - \bf{y}}
        \right)
        =
        \bf{0}.
    \end{gather}
\end{remark}
\section{Numerical implementation}
\label{sec:implementation}
In this section, we present a numerical implementation to efficiently compute approximations to the solutions of FIVP \eqref{eq:ivp}.
As a result, we obtain the approximating splines sequences as constructed in Theorem~\ref{thm:hilf_splines} by operationalizing the splines and fractional integration operators.

Given a function $f:[\eps,T]\to \mathbb{R}$ and a size $k+1$ knot collection $\mathcal{A}$ on $[\eps, T]$, the splines approximation operator can directly be implemented from Definition \ref{def:bernstein_operator} and \ref{def:spline_operator}:
\begin{align}
    \spline{q} 
    f (t)
    &= 
    \sum_{i=0}^k
    \sum_{j=0}^q
    \sum _{\ell=j }^{q}
    a_{i,j}
    {\binom {q}{\ell }}{\binom {\ell }{j }}(-1)^{\ell -j }
    \,
    \One_{A_i}
    \,
    s_{i}(t)^{\ell },
    \label{eq:spline_full_bernstein_formulation}
\end{align}
where $a_{i,j} = f(t_i + j (t_{i+1}-t_i)/q)$ and $s_{i} : A_i \to [0,1]$ is given by $s_{i}(t) := \fr{t - t_i}{t_{i+1}-t_i}$.

Then, for a given Bernstein spline $S : [\eps,T]\to \mathbb{R}$ with Bernstein coefficients $a_{i,j}$, we derive a closed-form expression for fractional integration.
First, we consider the expression $\One_{A_i}(t)\,
    s_i(t)^{\ell }$ inside of the summation of \eqref{eq:spline_full_bernstein_formulation}.
For $t < t_i$, the integral must be zero as the fractional integral of the zero function is zero.
Now, for $t\geq t_i$, the integral can be written out as follows:
\begin{align}
    \rlint{\eps}{t}{\alpha}
    \{
    \One_{A_i}(t)\,
    s_i(t)^{\ell }
    \}(t)
    &=
    \fr{1}{\Gamma(\alpha)}
    \int_{\eps}^{t}
    (t-x)^{\alpha-1}
    \One_{A_i}(x)\,
    s_i(x)^{\ell }
    \dto x
    \\
    &=
    \fr{1}{\Gamma(\alpha)}
    \int_{t_i}^{t \land t_{i+1}}
    (t-x)^{\alpha-1}
    s_i(x)^{\ell }
    \dto x
    \\
    &=
    \fr{(t_{i+1}-t_i)}{\Gamma(\alpha)}
    \int_{0}^{s_i(t)\land 1}
    (t-s_i^{-1}(z))^{\alpha-1}
    z^{\ell }
    \dto z \label{eq:bs_integral_trans_1}
    \\
    &=
    \fr{(t_{i+1}-t_i)^\alpha}{\Gamma(\alpha)}
    \int_{0}^{s_i(t)\land 1}
    \left(
    s_i(t)
    -z
    \right)^{\alpha-1}
    z^{\ell }
    \dto z.
\end{align}
Here, $a \land b:=\min(a,b)$ and we take the substitution $z := s_i(x)$.
The notation $s_i^{-1} : [0,1] \to [t_i, t_{i+1}]$ is used for the inverse $s_i^{-1}(z) = (t_{i+1}-t_i)z+t_i$.
Finally, using $\vartheta := z / s_i(t)$ gives:
\begin{align}
\rlint{\eps}{t}{\alpha}
    \{\One_{A_i}(t)\,
    s_i(t)^{\ell }\}(t)
    &=
    \fr{(t_{i+1}-t_i)^\alpha}{\Gamma(\alpha)}
    \int_{0}^{s_i(t)\land 1}
    \left(
    s_i(t)
    -z
    \right)^{\alpha-1}
    z^{\ell }
    \dto z
    \\
    &=
    \fr{(t_{i+1}-t_i)^\alpha s_i(t)^{\alpha+\ell}}{\Gamma(\alpha)}
    \int_{0}^{1 \land 1/s_i(t)}
    \left(1-\vartheta
    \right)^{\alpha-1}
    \vartheta^{\ell }
    \dto \vartheta
    \\
     &=
    \fr{(t-t_{i})^{\alpha+\ell}}{(t_{i+1}-t_i)^\ell \,\Gamma(\alpha)}
    \mathrm{B}_{1 \land 1/s_i(t)}(\ell+1, \alpha).\label{eq:int_res_spline}
\end{align}
Now, substituting \eqref{eq:int_res_spline} in \eqref{eq:spline_full_bernstein_formulation}, we have, by linearity of integration:
\begin{align}
\label{eq:spline_int_worked_out}
    \rlint{\eps}{t}{\alpha}
    S(t)
    =
    \sum_{i=0}^k
    \sum_{j=0}^q
    \sum _{\ell=j }^{q}
    a_{i,j}
    {\binom {q}{\ell }}{\binom {\ell }{j }}
    \fr{(t-t_{i})^{\alpha+\ell} \, \mathrm{B}_{1 \land 1/s_i(t)}(\ell+1, \alpha) }
    {(-1)^{j-\ell }(t_{i+1}-t_i)^\ell \,\Gamma(\alpha)}.
\end{align}
In order to efficiently obtain the numerical solutions constructed in Theorem~\ref{thm:hilf_splines}, we provide a vectorized software implementation.
Given a spline approximation of order $q$ such as obtained in \eqref{eq:spline_full_bernstein_formulation} with corresponding coefficient matrix $A^\supnum{q}$:
\begin{align}
    {A}^\supnum{q}_{i,j}
    =
    a_{i,j}, 
    \quad i \in \{0, ..., k\}, 
    \quad j \in \{0, ..., q\},
\end{align}
we are interested in computing the right hand side of expression \eqref{eq:spline_int_worked_out} for $p+1$ evaluation points $ 
\bf{\tilde{t}}=
\begin{bmatrix}
    \tilde{t}_0 & \tilde{t}_1 & ... & \tilde{t}_p
\end{bmatrix}$.
First, we vectorize the expression for the Bernstein basis coefficients by writing the size $(q+1) \times (q+1)$ matrix ${B}^\supnum{q}$ as:
\begin{align}
    B^\supnum{q}_{j, \ell}
    =
    \begin{cases}
        \binom{q}{\ell} \binom{\ell}{j}(-1)^{\ell-j}
        \quad & \textrm{ if } \ell \geq j,\\
        0 \quad & \textrm{ else},
    \end{cases}
    \quad \textrm{ for } 
    j, \ell \in \{0, ..., q\}.
    \label{eq:impl_B}
\end{align}
For the integration basis elements for a spline of order $q$ with evaluations at evaluation points $\bf{\tilde{t}}$, we construct a tensor ${J}^\supnum{q, \tilde{\bf{t}}}$ of size $(k+1) \times (q+1) \times (p+1)$, with entries defined as:
\begin{align}
\label{eq:impl_J}
    J^\supnum{q, \tilde{\bf{t}}}_{i, \ell, m}
    =
    \fr{(\tilde{t}_m-t_i)^{\alpha+\ell}\, \mathrm{B}_{1\land 
    1/s_i(\tau_m)
    }(\ell+1, \alpha)}
    {(t_{i+1}-t_i)^\ell \, \Gamma(\alpha)}
    ,
    \quad \textrm{ for } \quad
    \begin{cases}
    i \in \{0, ..., k\},\\ \ell \in \{0, ..., q\}, \\ 
    m \in \{0, ..., p\}.
    \end{cases}
\end{align}
For the numerical implementation of the incomplete beta function, we use the Python implementation as provided in the Scipy function \texttt{scipy.special.betainc}.
Now, the result of the fractional integral of spline $S$ at points $\tilde{t}_0, \ldots, \tilde{t}_p$ can be given as
\begin{align}
    \begin{bmatrix}
        \rlint{\eps}{\tilde{t}_0}{\alpha}
    S(\tilde{t}_0)
    &
    \rlint{\eps}{\tilde{t}_1}{\alpha}
    S(\tilde{t}_1)
    & 
    \cdots
    &
    \rlint{\eps}{\tilde{t}_p}{\alpha}
    S(\tilde{t}_p)
    \end{bmatrix}
    =
    ({A^\supnum{q}}{B^\supnum{q}})\cdot {J^\supnum{q, \tilde{\bf{t}}}},
    \label{eq:num_eval_points}
\end{align}
where the operation $\,\cdot\,$ denotes the tensor dot product with the first two dimensions of ${J}$, thus yielding a total resulting vector of size $p+1$ as required.
This tensor product can be efficiently implemented using paralellization.
In our implementation, we use the Python Numpy method \texttt{numpy.einsum}.
Note that to obtain coefficients $a_{i,j}$ for a spline of order $q$, we can construct an evaluation point matrix  $\tau_{i,j}$ of size $(k+1) \times (q+1)$:
\begin{align}
    \tau^\supnum{q}_{i,j}
    =
    t_i + j (t_{i+1}-t_i)/q
    , \quad
    i \in \{0, \ldots , k\},
    \quad
    j \in \{0, \ldots , q\}
    .
    \label{eq:impl_t}
\end{align}
We remark that the matrix $\tau^\supnum{q}$ corresponds to a vector of size $(k+1)(q+1)$ of the form $\tilde{\bf{t}}$ as used before.
In the pseudocode, we will use the notation $J^\supnum{q, \tau^\supnum{q}}$ to denote that the resulting evaluation points such as in \eqref{eq:num_eval_points} are recast to a matrix output of size $(k+1)\times (q+1)$.

\begin{algorithm}[p]
        \KwData{
            \begin{tabular}{ll}
            $\mathbf{f} : [0,T] \times \mathbb{R}^d \to \mathbb{R}^d$ & differential equation system \\  
            $\alpha \in (0,1)$ & Hilfer-derivative order \\  
            $\beta \in [0,1]$ & Hilfer-derivative type \\  
            $\mathbf{\tilde{y}}_0 \in \mathbb{R}^d$ & initial values \\  
            $\varepsilon \geq 0$ & time domain shift \\  
            $\mathcal{A}$ & knot collection of size $k+1$ on $[\varepsilon, T]$ \\  
            $q, q' \in \mathbb{N}$ & spline order, higher calculation spline order 
            \end{tabular}
        }
        \KwResult{=
            $\text{Spline solution $\bf{y}^{q,\eps}  : [\eps , T] \to \mathbb{R}^d$.}$
        }
            $\textrm{Initialize: } B^\supnum{q'}, \tau^\supnum{q}, \tau^\supnum{q'} \text{ and } J^\supnum{q', \tau^\supnum{q}} \,\, \text{(See \eqref{eq:impl_B} - \eqref{eq:impl_t})}$\;
            $A^\supnum{q}_{m, i, j} \gets 0
            \text{ for }
                m\in \{0, ..., d-1\}, 
                i \in \{0, ..., k\},
                j \in \{0, ..., q\}
            $\;
            $\gamma \gets \alpha+\beta-\alpha\beta $\;
            $\bf{\tilde{v}}_0 \gets 
            \bf{\tilde{y}}_0 / \Gamma(\gamma)$\;
        \For{$i \in \{0, \ldots, k\}$}{
        $\mathcal{I}_k \gets [(q+1)i, \cdots, (q+1)(i+1)]$\;
        \uIf{i = 0}{$A^\supnum{q}_{:, i, :} \gets \bf{\tilde{v}}_0$\;}
        \Else{$A^\supnum{q}_{:, i, :} \gets A^\supnum{q}_{:, i-1, q}$\;}
        \For{$m \in \{0, ..., d-1\}$}{
            $\tilde{A}^\supnum{q'}_\mathrm{b} \gets S^{q'}_{\mathcal{A}}\, \{f_m ( (\tau^\supnum{q}_{0:(i-1), :})^{\gamma-1} * {A}_{m,0:(i-1),:})\}
            (\tau^\supnum{q'}_{0:i-1, :})$\label{line:A_b}\;
            $\tilde{I}^\supnum{q}_\mathrm{b} \gets (\tilde{A}^\supnum{q'}_\mathrm{b} B^\supnum{q'}) \cdot J^\supnum{q', \tau^\supnum{q}}_{0:i-1, :, \mathcal{I}_k}$\;
            \label{line:I_b}
            \While{\label{line:conv} $! \texttt{ {\upshape  CONV}}$}{
              $\tilde{A}^\supnum{q'}_\mathrm{c} \gets S^{q'}_{\mathcal{A}}\, \{f_m ((\tau^\supnum{q}_{i, :})^{\gamma-1} * {A}_{m,i,:})\}(\tau^\supnum{q'}_{i, :})$\;
              $\tilde{I}^\supnum{q}_\mathrm{c} 
              \gets (\tilde{A}^\supnum{q'}_\mathrm{c}B^\supnum{q'}) \cdot J^{\supnum{q', \tau^\supnum{q}}}_{i,:,\mathcal{I}_k}$
              \label{line:I_c}\;
            $A^\supnum{q}_{m, i, :}
                \gets
                ({\tilde{v}}_{0})_m
                +
                (\tau^\supnum{q}_{i, :})^{1-\gamma}
                (\tilde{I}^\supnum{q}_\mathrm{b}+\tilde{I}^\supnum{q}_\mathrm{c})
            $\;
              }
          }
        }%
        \textbf{Return:}
        $y^{q,\eps}_m(t) \gets
        t^{\gamma-1}  \{S_\mathcal{A}^q {A}^\supnum{q}_{m, :, :}\}(t)$ 
\caption{Pseudocode of Hilfer fractional derivative splines approximation procedure.
}
\label{alg:pseudocode}
\end{algorithm}

In Algorithm \ref{alg:pseudocode}, we present the numerical procedure to to obtain solution approximations as in Theorem~\ref{thm:hilf_splines}.
A practical consideration is the method used to evaluate multiplication and function compositions $\bf{f}$ on splines.
In our implementation, we analytically compute splines multiplication.
This method has the benefit of not introducing additional approximation errors for multiplicative operations.
However, one downside to this method is that the product of $n$ splines of order $q$ yields a resulting spline of order $nq=:q'$. 
In the pseudocode of Algorithm \ref{alg:pseudocode}, we write $S^{q'}_\mathcal{A} \{ g * A\}$ for the product of a spline $g$ with one with coefficients $A$, resulting in a spline of order $q'$.
To prevent the spline order from increasing with every timestep, we implement a downscaling to order $q$ by evaluating the fractional integral at time evaluation points $\tau^\supnum{q}$. 
Hence, the resulting values correspond to the coefficient values for an order $q$ spline.
This is carried out in line \ref{line:I_b} and \ref{line:I_c} of the pseudocode, where the previous and current knot integral contributions are computed as $\tilde{I}_\mathrm{b}^\supnum{q}$ and $\tilde{I}_\mathrm{c}^\supnum{q}$ respectively, yielding an order $q$ spline at knot $k$.

In line \ref{line:conv}, \texttt{CONV} denotes a convergence criterion, which can be seen as the condition for truncation of the limit in the components of $\bf{n}$ in the analytical proof.
In practice, we can choose for instance a fixed number of iterations, or a maximum change $\eps_\textrm{it}$ in spline coefficients.
This results in a termination of Picard iterates at $n$ when 
$\sup_{j \in \{0, \ldots q\}} |(A_{m, i, j})_n - (A_{m, i,j})_{n-1}| < \eps_\textrm{it}$.
Finally, we remark that as the number of knots $k$ grows, the previous integral contribution calculations of lines \ref{line:A_b} - \ref{line:I_b} will dominate, assuming the number of iterations needed for convergence is independent of $k$.
{This leads to a time complexity of $\mathcal{O}(k^2)$, since for every knot $i \in \{0, \ldots, k\}$, the fractional integral is computed over values of the $i$ previous knots.
This highlights the ``memory effect'' of fractional differential equations.}
\section{Numerical results}
\label{sec:num_res}
\subsection{Polynomial example and convergence results}
\label{sec:num_res_conv}
To validate the analytical convergence results of Theorem \ref{thm:hilf_splines}, a numerical example is implemented for the following FIVP:
\begin{align}
    \begin{cases}
        \hilfder{0}{t}{\alpha, \beta}
        y(t) 
        =
        t^{k},
        \quad &t>0,
        \quad 0<\alpha<1,
        \quad 0\leq \beta \leq 1,
        \quad k>0,
        \\
        \rlint{0}{t}{1-\gamma}y(0)
        =
        \tilde{y}_0,
        \quad 
        &\tilde{y}_0 \in \mathbb{R}
        .
        \label{eq:num_pol_sys}
    \end{cases}
\end{align}
We take $\alpha, \beta = 1/2$, $k = 0.9$, $\tilde{y}_0 = 1$.
Approximation errors and convergence results are presented in Figure~\ref{fig:pol_num_res} and Table~\ref{tab:pol_h}~-~\ref{tab:pol_eps}.
\begin{figure}
    \centering
    \begin{subfigure}[t]{1\linewidth}        \includegraphics[width=\linewidth]{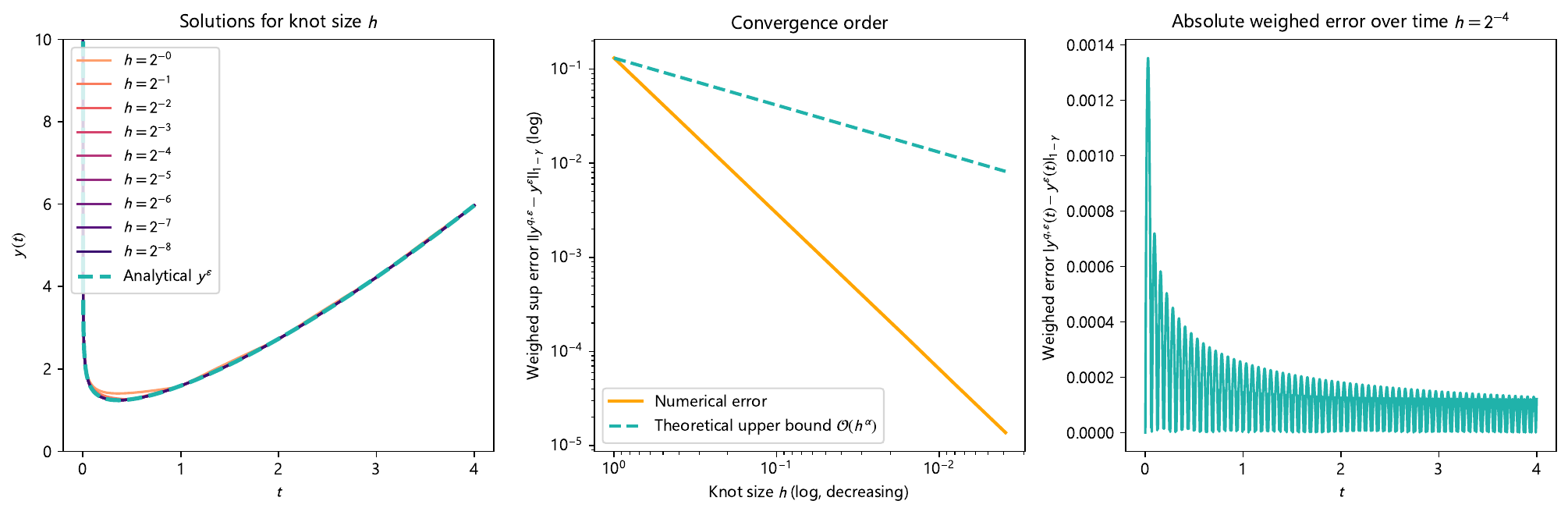}
        \caption{Convergence results in decreasing knot size $h$ with $q = 1$, $\eps = 10^{-10}.$}
    \end{subfigure}
    \begin{subfigure}[t]{1\linewidth}
        \includegraphics[width=\linewidth]{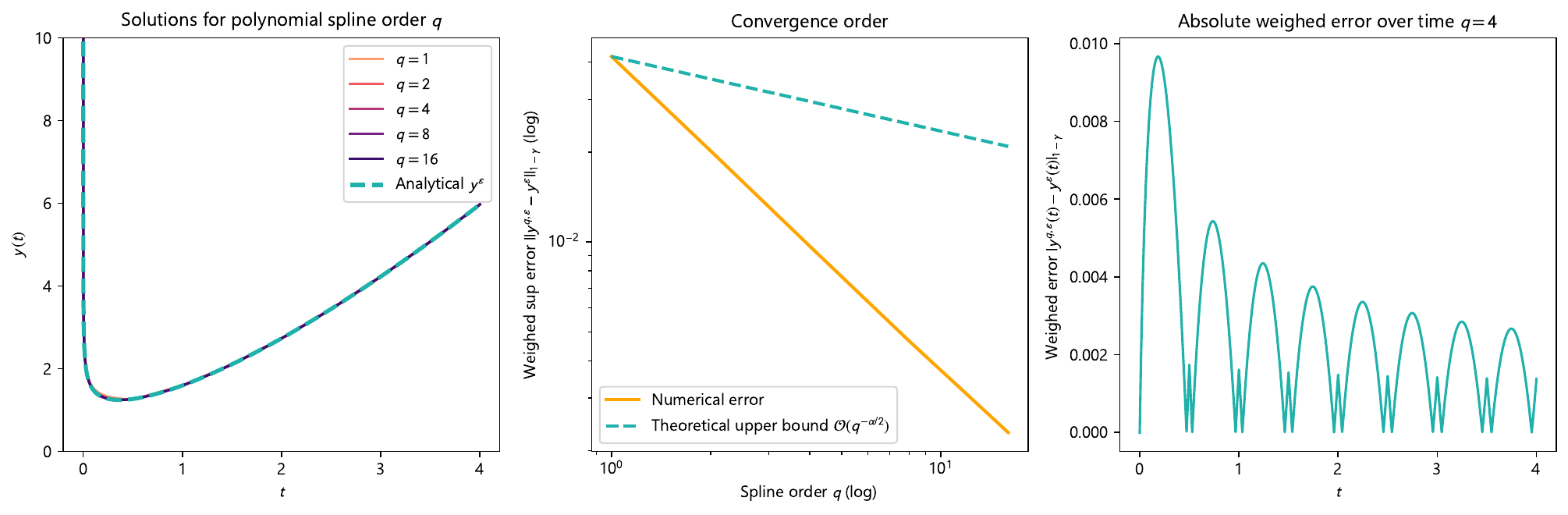}
        \caption{Convergence results in increasing polynomial order $q$ with $h = 1/2$, $\eps = 10^{-10}.$}
    \end{subfigure}
    \begin{subfigure}[t]{1\linewidth}
        \includegraphics[width=\linewidth]{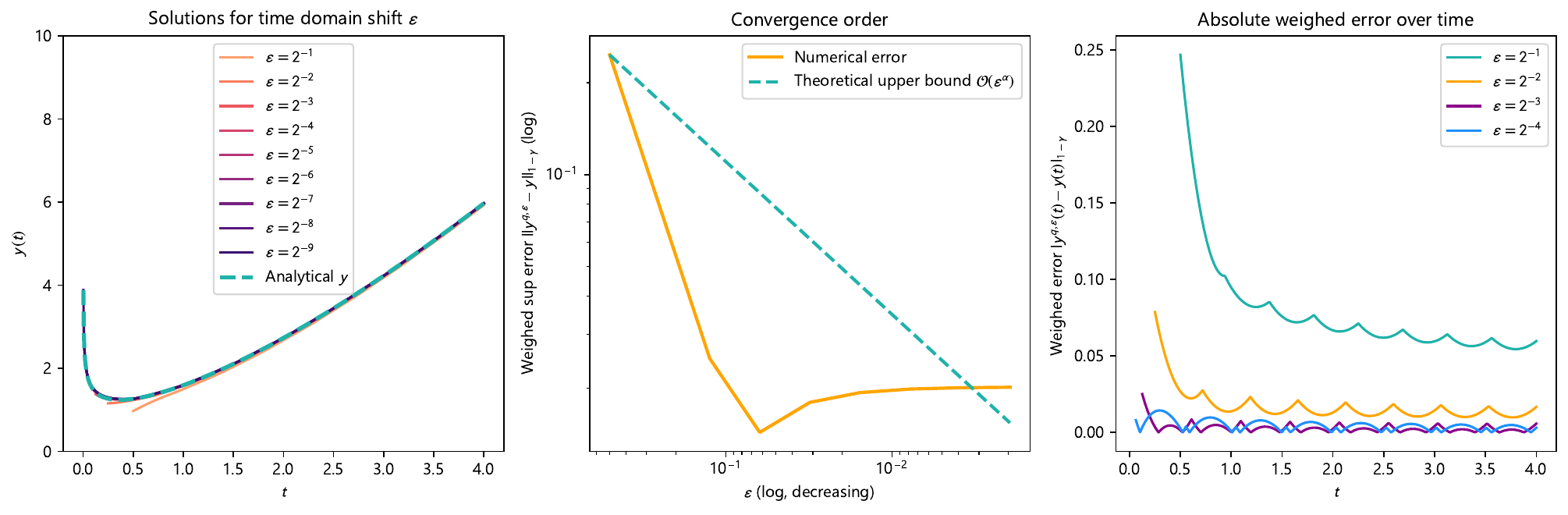}
        \caption{Convergence results in decreasing $\eps$ with $q = 2$, $h = 1/2$.}
    \end{subfigure}
    \caption{Results of convergence in $q$, $h$ and $\eps$ for system \eqref{eq:num_pol_sys} with $\alpha, \beta = 0.5$, $k=0.9$, $\tilde{y}_0 = 1$.
    }
    \label{fig:pol_num_res}
\end{figure}

\begin{table}[h]
    \centering
    \begin{tabular}{c|c|c|c}
         $h$ & mean weighted error & sup weighted error & total time (s) \\ \hline
$2^{-0}$ &$ 4.747 \cdot 10^{-2 }$& $1.312 \cdot 10^{-1}$ & $6.082 \cdot 10^{-3}$ \\
$2^{-1}$ &$ 1.123 \cdot 10^{-2 }$& $4.181 \cdot 10^{-2}$ & $2.622 \cdot 10^{-3}$ \\
$2^{-2}$ &$ 2.639 \cdot 10^{-3 }$& $1.332 \cdot 10^{-2}$ & $3.657 \cdot 10^{-3}$ \\
$2^{-3}$ &$ 6.164 \cdot 10^{-4 }$& $4.244 \cdot 10^{-3}$ & $4.793 \cdot 10^{-3}$ \\
$2^{-4}$ &$ 1.433 \cdot 10^{-4 }$& $1.352 \cdot 10^{-3}$ & $8.702 \cdot 10^{-3}$ \\
$2^{-5}$ &$ 3.322 \cdot 10^{-5 }$& $4.309 \cdot 10^{-4}$ & $1.474 \cdot 10^{-2}$ \\
$2^{-6}$ &$ 7.692 \cdot 10^{-6 }$& $1.373 \cdot 10^{-4}$ & $2.949 \cdot 10^{-2}$ \\
$2^{-7}$ &$ 1.781 \cdot 10^{-6 }$& $4.316 \cdot 10^{-5}$ & $5.849 \cdot 10^{-2}$ \\
$2^{-8}$ &$ 4.403 \cdot 10^{-7 }$& $1.375 \cdot 10^{-5}$ & $1.143 \cdot 10^{-1}$ \\
\hline
    \end{tabular}
    \caption{Convergence results for system \eqref{eq:num_pol_sys} in decreasing knot size $h$ with $q = 1$, $\eps = 10^{-10}.$}
    \label{tab:pol_h}
\end{table}
\begin{table}[h]
    \centering
    \begin{tabular}{c|c|c|c}
$q$ & mean weighted error & sup weighted error & total time (s) \\ \hline
$1$ &$ 1.123 \cdot 10^{-2 }$& $4.181 \cdot 10^{-2}$ & $7.325 \cdot 10^{-3}$ \\
$2$ &$ 5.424 \cdot 10^{-3 }$& $2.018 \cdot 10^{-2}$ & $2.944 \cdot 10^{-3}$ \\
$4$ &$ 2.671 \cdot 10^{-3 }$& $9.670 \cdot 10^{-3}$ & $4.752 \cdot 10^{-3}$ \\
$8$ &$ 1.331 \cdot 10^{-3 }$& $4.660 \cdot 10^{-3}$ & $3.375 \cdot 10^{-3}$ \\
$16$ &$ 6.675 \cdot 10^{-4 }$& $2.298 \cdot 10^{-3}$ & $4.713 \cdot 10^{-3}$ \\
\hline
    \end{tabular}
    \caption{Convergence results for system \eqref{eq:num_pol_sys} in increasing polynomial order $q$ with $h = 1/2$, $\eps = 10^{-10}$.}
    \label{tab:pol_q}
\end{table}
\begin{table}[h]
    \footnotesize
    \centering
    \begin{tabular}{c|c|c|c|c}
$\varepsilon$ & mean weighted error & sup weighted error & total time (s) & $x^{q,\varepsilon}(\varepsilon)$ \\ \hline
$2^{-1}$ &$ 2.467 \cdot 10^{-1 }$& $2.467 \cdot 10^{-1}$ & $6.505 \cdot 10^{-3}$ & $9.705 \cdot 10^{-1}$ \\
$2^{-2}$ &$ 7.861 \cdot 10^{-2 }$& $7.861 \cdot 10^{-2}$ & $2.621 \cdot 10^{-3}$ & $1.154 \cdot 10^{0}$ \\
$2^{-3}$ &$ 2.505 \cdot 10^{-2 }$& $2.505 \cdot 10^{-2}$ & $3.290 \cdot 10^{-3}$ & $1.372 \cdot 10^{0}$ \\
$2^{-4}$ &$ 1.435 \cdot 10^{-2 }$& $1.435 \cdot 10^{-2}$ & $2.544 \cdot 10^{-3}$ & $1.632 \cdot 10^{0}$ \\
$2^{-5}$ &$ 1.800 \cdot 10^{-2 }$& $1.800 \cdot 10^{-2}$ & $2.485 \cdot 10^{-3}$ & $1.941 \cdot 10^{0}$ \\
$2^{-6}$ &$ 1.937 \cdot 10^{-2 }$& $1.937 \cdot 10^{-2}$ & $2.486 \cdot 10^{-3}$ & $2.308 \cdot 10^{0}$ \\
$2^{-7}$ &$ 1.990 \cdot 10^{-2 }$& $1.990 \cdot 10^{-2}$ & $2.875 \cdot 10^{-3}$ & $2.745 \cdot 10^{0}$ \\
$2^{-8}$ &$ 2.010 \cdot 10^{-2 }$& $2.010 \cdot 10^{-2}$ & $2.557 \cdot 10^{-3}$ & $3.264 \cdot 10^{0}$ \\
$2^{-9}$ &$ 2.017 \cdot 10^{-2 }$& $2.017 \cdot 10^{-2}$ & $3.135 \cdot 10^{-3}$ & $3.882 \cdot 10^{0}$ \\
\hline
    \end{tabular}
    \caption{Convergence results for system \eqref{eq:num_pol_sys} in decreasing $\eps$ with $q = 2$, $h = 1/2$.}
    \label{tab:pol_eps}
\end{table}

Results demonstrate convergence orders satisfying the upper bounds of Theorem \ref{thm:hilf_splines} in polynomial order $q$, knot size $h$ and time-shift $\eps$, as can be seen in Figure~\ref{fig:pol_num_res}.
Convergence is in general faster compared to the theoretical upper bound, possibly caused by higher regularity of solutions.
In the convergence results of $h$ and $q$, the error relative to $y^\eps$ can be seen to attain minima at knot intersections, with higher errors inside the knot intervals.
We note that this corresponds with the theoretical results of e.g. Remark \ref{rem:bernstein_endpoints}.

Moreover, for convergence in $q$, we remark that higher values do not converge due to numerical stability issues, likely caused by Runge's Phenomenon (see e.g. \cite{Schumaker2007}).
Increasing $q$ when possible however does provide a practical benefit.
As can be seen in Table~\ref{tab:pol_q}, the increase of polynomial order $q$ does not lead to an increase in total run time, contrary to reducing the knot size $h$ (Table~\ref{tab:pol_h}).
We conjecture that this is due to the fact that the polynomial basis function effects can be computed in parallel, operationalized in implementation by the tensor product vectorization setup as described in Section \ref{sec:implementation}.  
This provides a promising avenue for computational speedup of the algorithm in practical use-cases.

For the highest values of $\eps$, the error at $t\downarrow \eps$ dominates, decreasing when decreasing $\eps$ (Table~\ref{tab:pol_eps}).
In these cases, the errors appear to be highest at knot intersections, contrary to the other examples.
From $\eps = 2^{-4}$ however, lowering $\eps$ does not appear to yield a significant improvement in error metric as the error caused by splines approximations appears to dominate over the error introduced by time-shifting $\eps$.
We note that this is an indication that for most practical problems, further lowering $\eps$ will not yield a significant improvement in the overall error and behavior of the solution.
However, in cases where $y^{q, \eps}(\eps)$ is required to be ``closer'' to the limiting case of singular behavior at $t \downarrow 0$, it can still be desirable to choose a lower value of $\eps$. 

\subsection{Comparison with Predictor-Corrector method}
\label{sec:num_res_diethelm}
We compare our splines method with the commonly used Adams–\allowbreak~Bashforth–\allowbreak~Moulton Predictor-Corrector method of \citet{Diethelm2002}.
Since this method is only suited for the Caputo case, we consider only $\beta=1$.
Take the following fractional system:
\begin{align}
    \begin{cases}
        \label{eq:diethelm_sys_lin_a}
        \hilfder{0}{t}{\alpha, 1} x(t)
        =
        a x(t),
        \quad &t>0,
        \quad 0<\alpha<1,
        \quad a\in \mathbb{R},
        \\
        x(0) = {\tilde{x}_0},
        \quad &{\tilde{x}_0}\in\mathbb{R}
        .
    \end{cases}
\end{align}
This system has the known analytical solution $x(t) = \tilde{x}_0 \mathrm{E}_\alpha(at^\alpha)$,
where $\mathrm{E}_\alpha(z) = \sum_{j=0}^\infty \fr{z^j}{\Gamma(j\alpha+1)}$ denotes the Mittag-Leffler function.
In this section, we will take $\alpha = 0.5, a = -1$ and $\tilde{x}_0 = 1$.

\begin{figure}[h]
    \centering
    \includegraphics[width=1\linewidth]{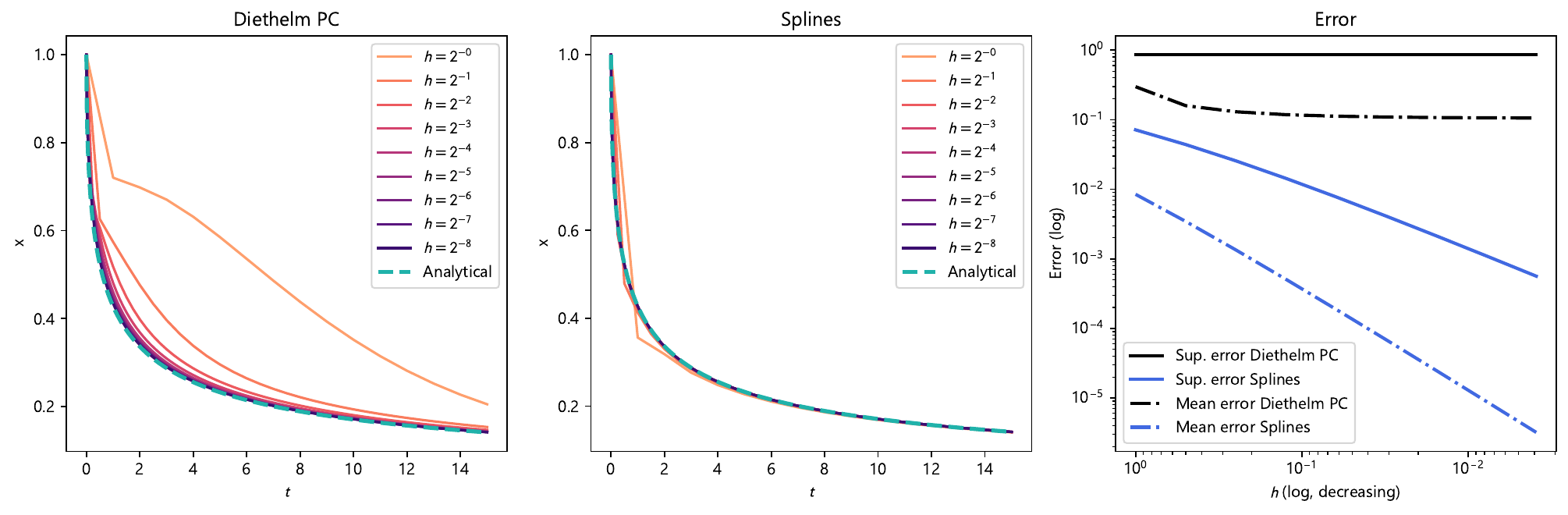}
    \caption{Comparison of approximations of the predictor-corrector method (left) with our Bernstein splines setup (middle) and respective errors (right) for decreasing knot size values $h$ for system \eqref{eq:diethelm_sys_lin_a}.}
    \label{fig:diethelm_h}
\end{figure}

\begin{table}[h]
    \centering
    \scriptsize
    \begin{tabular}{c||c|c|c||c|c|c}
$h$ & mean error PC & sup error PC & time PC (s) & mean error BS & sup error BS & time BS (s)  \\ \hline
$2^{-0}$ & $2.986 \cdot 10^{-1 } $ & $8.588 \cdot 10^{-1 } $ & $5.370 \cdot 10^{-4 } $ & $8.465 \cdot 10^{-3 } $ & $7.154 \cdot 10^{-2 } $ & $4.786 \cdot 10^{-2 } $\\
$2^{-1}$ & $1.594 \cdot 10^{-1 } $ & $8.588 \cdot 10^{-1 } $ & $9.712 \cdot 10^{-4 } $ & $3.452 \cdot 10^{-3 } $ & $4.400 \cdot 10^{-2 } $ & $4.165 \cdot 10^{-2 } $\\
$2^{-2}$ & $1.302 \cdot 10^{-1 } $ & $8.588 \cdot 10^{-1 } $ & $2.562 \cdot 10^{-3 } $ & $1.346 \cdot 10^{-3 } $ & $2.567 \cdot 10^{-2 } $ & $5.640 \cdot 10^{-2 } $\\
$2^{-3}$ & $1.183 \cdot 10^{-1 } $ & $8.588 \cdot 10^{-1 } $ & $4.171 \cdot 10^{-3 } $ & $5.088 \cdot 10^{-4 } $ & $1.437 \cdot 10^{-2 } $ & $7.901 \cdot 10^{-2 } $\\
$2^{-4}$ & $1.124 \cdot 10^{-1 } $ & $8.588 \cdot 10^{-1 } $ & $7.240 \cdot 10^{-3 } $ & $1.884 \cdot 10^{-4 } $ & $7.787 \cdot 10^{-3 } $ & $1.233 \cdot 10^{-1 } $\\
$2^{-5}$ & $1.093 \cdot 10^{-1 } $ & $8.588 \cdot 10^{-1 } $ & $1.553 \cdot 10^{-2 } $ & $6.880 \cdot 10^{-5 } $ & $4.123 \cdot 10^{-3 } $ & $2.152 \cdot 10^{-1 } $\\
$2^{-6}$ & $1.075 \cdot 10^{-1 } $ & $8.588 \cdot 10^{-1 } $ & $3.722 \cdot 10^{-2 } $ & $2.488 \cdot 10^{-5 } $ & $2.146 \cdot 10^{-3 } $ & $3.786 \cdot 10^{-1 } $\\
$2^{-7}$ & $1.065 \cdot 10^{-1 } $ & $8.588 \cdot 10^{-1 } $ & $7.533 \cdot 10^{-2 } $ & $8.934 \cdot 10^{-6 } $ & $1.104 \cdot 10^{-3 } $ & $6.338 \cdot 10^{-1 } $\\
$2^{-8}$ & $1.058 \cdot 10^{-1 } $ & $8.588 \cdot 10^{-1 } $ & $1.441 \cdot 10^{-1 } $ & $3.194 \cdot 10^{-6 } $ & $5.631 \cdot 10^{-4 } $ & $1.451 \cdot 10^{0 } $\\
\hline
    \end{tabular}
    \caption{Errors of the predictor-corrector method (PC) and our Bernstein splines setup (BS) for decreasing knot size values $h$ for system \eqref{eq:diethelm_sys_lin_a}.}
    \label{tab:diethelm_h}
\end{table}

First, we obtain results when decreasing knot size $h$ in an equidistant grid for $0\leq t \leq 15$.
Results are presented in Figure~\ref{fig:diethelm_h} and Table~\ref{tab:diethelm_h}.
We note that the predictor-corrector method can be seen as an order one splines method.
Hence, we choose $q=1$ for comparison.
Errors are calculated at knot points.
The analytical Mittag-Leffler solution is computed up to $500$ summation terms, and an iteration convergence tolerance of $10^{-12}$ is chosen.
Since we are dealing with the Caputo case, we can take $\eps=0$ for the time domain.

Results show a lower error for the splines method as compared to the predictor-corrector method for every choice of $h$.
Moreover, the supremum error in the predictor-corrector does not decrease further when decreasing $h$ in our setup.
Graphically, this can be explained by a persistent ``bias'' in function approximation of the predictor-corrector step, yielding a higher value of numerical solutions as compared to the analytical solutions.
We conjecture that this effect is mitigated in the splines setup since the fractional integral over past values is not approximated by numerical integration rules, but carried out analytically over the splines function.
However, the predictor-corrector approach does outperform the splines setup for the same values of $h$ in terms of computational time required.

\begin{table}[h]
    \centering
    \begin{tabular}{c|c|c|c}
& mean error & sup error & total time (s) \\ \hline
Diethelm PC & $1.112 \cdot 10^{-1 }$ & $8.588 \cdot 10^{-1 }$ & $9.347 \cdot 10^{-3}$ \\ \hline
Splines $N = 1$ &$ 1.366 \cdot 10^{-3 }$& $4.269 \cdot 10^{-2}$ & $9.136 \cdot 10^{-2}$ \\
Splines $N = 2$ &$ 1.083 \cdot 10^{-4 }$& $7.388 \cdot 10^{-3}$ & $1.102 \cdot 10^{-1}$ \\
Splines $N = 3$ &$ 1.665 \cdot 10^{-4 }$& $6.355 \cdot 10^{-3}$ & $1.127 \cdot 10^{-1}$ \\
Splines $N = 4$ &$ 1.313 \cdot 10^{-4 }$& $6.360 \cdot 10^{-3}$ & $1.290 \cdot 10^{-1}$ \\
Splines $N = 5$ &$ 1.373 \cdot 10^{-4 }$& $6.360 \cdot 10^{-3}$ & $1.136 \cdot 10^{-1}$ \\
Splines $N = 6$ &$ 1.363 \cdot 10^{-4 }$& $6.360 \cdot 10^{-3}$ & $1.300 \cdot 10^{-1}$ \\
Splines $N = 7$ &$ 1.364 \cdot 10^{-4 }$& $6.360 \cdot 10^{-3}$ & $1.407 \cdot 10^{-1}$ \\
Splines $N = 8$ &$ 1.364 \cdot 10^{-4 }$& $6.360 \cdot 10^{-3}$ & $1.555 \cdot 10^{-1}$ \\
Splines $N = 9$ &$ 1.364 \cdot 10^{-4 }$& $6.360 \cdot 10^{-3}$ & $1.775 \cdot 10^{-1}$ \\
Splines $N = 10$ &$ 1.364 \cdot 10^{-4 }$& $6.360 \cdot 10^{-3}$ & $1.877 \cdot 10^{-1}$ \\
\hline
    \end{tabular}
    \caption{Errors of the predictor-corrector method and our Bernstein splines setup when increasing the number of knot iterations $N$ for system \eqref{eq:diethelm_sys_lin_a}.}
    \label{tab:diethelm_N}
\end{table}

Finally, we remark that the predictor-corrector setup can be seen as a two-step iteration method.
This approach is notably different from the results as above, where the Bernstein splines method is iterated $N$ times per knot until the convergence criterion is satisfied.
To showcase the error for a restricted number of iterations $N$ per knot, we fix $h=0.05$ and obtain results when increasing $N$.
Results are presented in Table~\ref{tab:diethelm_N}.
We observe that even with only one integration iteration, the obtained errors are significantly lower using the splines method as compared to the predictor-corrector method.
Moreover, the error stabilizes fast, yielding no large change in error after only a few iterations.
Computation times are notably lower for the predictor-corrector method, although a similar error level cannot be reached.
\subsection{Fractional Van der Pol oscillator}
\label{sec:num_res_vdp}
Consider now the Hilfer-fractional Van der Pol oscillator:
\begin{align}
    \ddot{x}(t)
    -
    \mu (1-x(t)^2) 
    \hilfder{0}{t}{\alpha, \beta} x(t)
    +
    x(t) 
    =0,
    \quad \mu > 0.
    \label{eq:vdp_one_eq}
\end{align}
To write this in a system notation satisfying \eqref{eq:ivp}, note that by the component-wise equivalent integral equation \eqref{eq:eq_int_eq} and the semigroup property from Proposition \ref{prop:rlint_semigroup}, we can write \eqref{eq:vdp_one_eq} as a combination of fractional derivatives with $0 <\alpha < 1$.
Take $\alpha =1/2$, hence  $\gamma = \alpha+1/2-1/2 \beta$.
Then, taking:
$\bf{x}(t) = (x(t), y(t), z(t), u(t))$,
gives the following FIVP  for $\tilde{x}_0  \in \mathbb{R}$:
\begin{align}
\label{eq:res_vdp_sys}
     \hilfder{0}{t}{1/2, \,\beta} 
     \bf{x}(t)
     =
     \begin{pmatrix}
         y(t)\\
         z(t)\\
         u(t)\\
          \mu (1-x(t)^2) y(t) - x(t)
     \end{pmatrix},
     \quad
     \rlint{0}{t}{1-\gamma} \bf{x}(0)
        =
        \begin{pmatrix}
         \tilde{x}_0\\
         0\\
         0\\
         0
     \end{pmatrix}.
\end{align}
System \eqref{eq:res_vdp_sys} is then approximated using the Hilfer-derivative method of Section \ref{sec:implementation}.
A time-shift value of $\eps = 10^{-5}$ is chosen by empirical convergence requirements, as will be discussed below.
Knots are constructed so that the knot size $h_i:= t_{i+1} - t_i$ satisfies
\begin{align}
    \label{eq:vdp_knot_sel}
    h_i 
    =
    \begin{cases}
    \min\left(h_{\textrm{max}},\, c^{1/(1-\gamma)-1}t_i\right),
    \quad     &\text{if} \quad 0\leq \gamma < 1,
    \\
    h_{\textrm{max}} \quad &\text{if} \quad \gamma = 1,
    \end{cases}
\end{align}
for $i\in\{0, ..., k\}$.
This choice of $h_i$ ensures the uniform knot convergence requirements of Theorem \ref{thm:hilf_splines} for $\left({1+{h_i}/{t_i}}\right)^{1-\gamma} \leq c$.
For implementation, values of $c = \fr{3}{2}$ and $h_{\textrm{max}} = 0.05$ are chosen, motivated by plotting resolution.
Since the highest order of spline multiplication in system \eqref{eq:res_vdp_sys} is an order $3$ multiplication for $x^2(t)y(t)$, Bernstein splines of order $q$ can be implemented with an intermediate computational order $q' = 3q$.
Simulations are run using $q = 1$ for $T = 100$ and ${\tilde{x}}_0 = 1$ using various values of $\beta \in [0,1]$.
Run statistics are presented in Table \ref{tab:res_vdp_runs} with a graphical overview of results in Figure \ref{fig:res_vdp}.

\begin{table}[!htbp]
    \centering
    \footnotesize
    \begin{tabular}{c|c|c|c|c|c}
$\beta$ & knots & $x(\eps)$ & total time (s) &  avg. it. per knot & avg. time per it. (s) \\ \hline
$0.00$ & 2010 & $1.784 \cdot 10^2 $& $7.187$ & $20.763$&$ 1.722 \cdot 10^{-4 }$\\
$0.25$ & 2007 & $5.227 \cdot 10^1 $& $8.390$ & $20.418$&$ 2.047 \cdot 10^{-4 }$\\ 
$0.50$ & 2005 & $1.451 \cdot 10^1 $& $6.679$ & $20.063$&$ 1.660 \cdot 10^{-4 }$\\
$0.75$ & 2002 & $3.870 \cdot 10^0 $& $6.655$ & $19.755$&$ 1.683 \cdot 10^{-4 }$\\ 
$1.00$ & 2000 & $1.000 \cdot 10^0 $& $7.652$ & $19.478$&$ 1.964 \cdot 10^{-4 }$\\ \hline
    \end{tabular}
    \caption{Run statistics for simulations of the fractional Van der Pol oscillator \eqref{eq:res_vdp_sys} for $T = 100$.}
    \label{tab:res_vdp_runs}
\end{table}
\begin{figure}
    \begin{subfigure}[t]{0.420\linewidth}
        \centering
        \includegraphics[width=\linewidth]{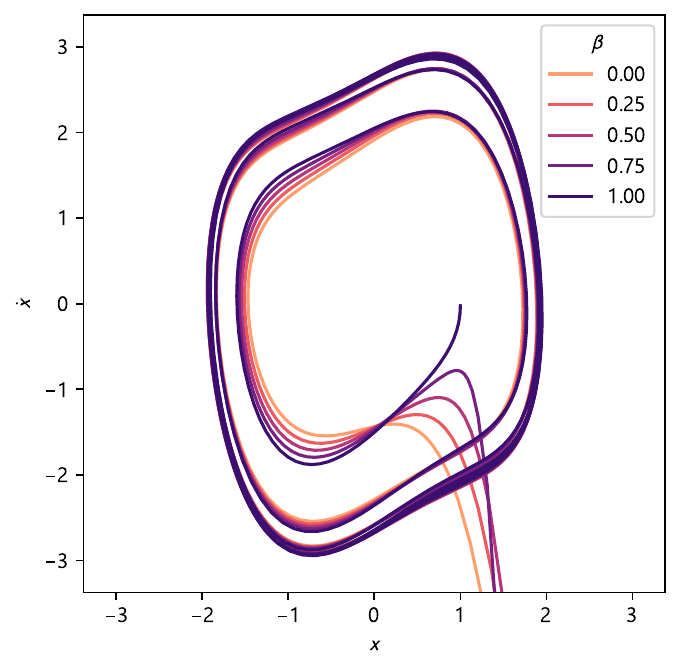}
        \caption{Phase portrait $(T=100)$}
    \end{subfigure}
    \begin{subfigure}[t]{0.580\linewidth}
        \centering
        \includegraphics[width=\linewidth]{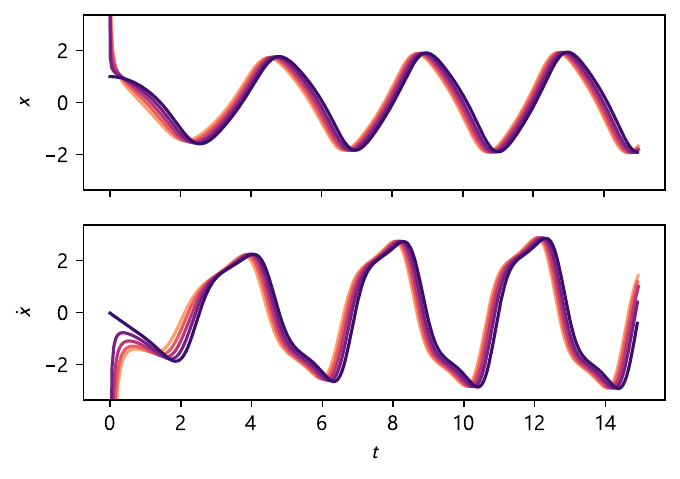}
        \caption{Solution $x$ and derivative $\dot{x}$ $(T=15)$.}
    \end{subfigure}
    \caption{Results for the Hilfer-derivative fractional Van der Pol oscillator \eqref{eq:res_vdp_sys} for ${\tilde{x}}_0 = 1$, $\alpha = 1/2$ and various $\beta$.}
    \label{fig:res_vdp}
\end{figure}

\newpage
Results show convergence for the numerical method with a consistently about $20$ iterations needed per knot across all values of $\beta$, indicating that the knot selection \eqref{eq:vdp_knot_sel} provides satisfactory convergent results. 
Obtained run times are similar, indicating minimal performance difference when dealing with the non-Caputo cases.
Finally, we note that a higher value of $x(\eps)$ can be obtained by lowering the value of $\eps > 0$.
However, we found empirically that raising $\eps$ can cause numerical divergence for low values of $\beta$, possibly due to rounding errors in the $t^{\gamma-1}$ compositions.
We note that for the fractional Van der Pol system, the overall solution behavior was observed to be unaffected when lowering $\eps$ whenever possible.

In terms of solution behavior and stability, Figure \ref{fig:res_vdp} shows solutions converging to the same limit cycles independent of Hilfer derivative type $\beta$, indicating strong attraction to the limit cycle.
Viewing the Hilfer-derivative setting as a singular, decaying perturbation at $t=0$ of the Caputo-derivative system, this convergence to a unique limit cycle corresponds to semi-analytical results as obtained by \cite{Liu2016vibcont}.
\section{Conclusions}
\label{sec:conclusion}
In this paper, we present a numerical Bernstein splines technique to obtain solutions to nonlinear Hilfer-derivative FIVPs of the form of \eqref{eq:ivp}.
The approximations are constructed using Picard iterations on each knot of the locally supported spline polynomials.
Analytical convergence requirements of the constructed approximations are presented in Theorem~\ref{thm:hilf_splines} and asymptotic convergence rates to the analytical solutions of \eqref{eq:ivp} are presented in Theorem~\ref{thm:convergence_rates}.
The approximations are obtained for $t>\eps$ in the weighted continuous space $C_{1-\gamma}[\eps, T]$, converging to the analytical solution as $\eps \downarrow 0$.
To deal with the singular behavior of solutions as $t\downarrow0$, a transformation is made to the weighted space, transforming back after convergence.

We present a vectorized implementation, providing improved numerical performance and parallelization in polynomial order $q$.
Numerical experiments show convergence rates corresponding to the analytical results.
Moreover, our method results in significantly lower errors compared to the commonly used Adams–Bashforth–Moulton predictor-corrector method of \citet{Diethelm2002}, even for large knot sizes and a low number of Picard iterations.
Finally, we demonstrate the applicability of our method to a nonlinear example, simulating solutions of the Hilfer-fractional Van der Pol oscillator.

For future work, we note that our splines method, although significantly more accurate, is still more computationally demanding compared to the predictor-corrector method.
More avenues could be explored to make the splines method implementation more efficient.
Additionally, different spline basis functions can be explored, possibly extending the method to implementations in (fractionally) differentiable solution spaces.
Moreover, we wish to extend the method from initial to more general boundary value problems, both in an ordinary FDE-setting and with applications to FPDEs.
Here, the splines structure could possibly be utilized to construct a space-like domain.
\section*{Data availability}
Code and data is available on \url{https://github.com/ngoedegebure/fractional_splines}.
The code used in Section \ref{sec:num_res} can be found under \texttt{/examples/IVP-paper/} .
\bibliographystyle{unsrtnat} 
\bibliography{references.bib}

\end{document}